\documentclass[12pt,reqno]{amsproc}%
\usepackage{amsfonts}
\usepackage{amsmath}
\usepackage{amssymb}
\usepackage{graphicx}%
 \theoremstyle{plain}

\newtheorem{claim}{Claim} [subsection]

\newtheorem{corollary}{Corollary} [subsection]

\newtheorem{definition}{Definition} [subsection]
\newtheorem{example}{Example} [subsection]

\newtheorem{lemma}{Lemma}  [subsection]

\newtheorem{proposition}{Proposition} [subsection]
\newtheorem{remark}{Remark} [subsection]

\newtheorem{theorem}{Theorem} [subsection]
 \numberwithin{equation}  {subsection}
\begin{document}
\centerline{\Large\bf{Reduced Free Products of Unital AH Algebras}}

\vspace{0.5cm}

\centerline{\Large\bf{and}}

\vspace{0.5cm}

\centerline{\Large\bf{Blackadar and Kirchberg's MF Algebras}}

\vspace{1cm}



 \centerline{Don Hadwin \qquad Jiankui Li\footnote{The  second author is partially
 supported by an NSF of China.} \qquad   Junhao Shen\footnote{The  third author is partially
 supported by an NSF grant.} \qquad Liguang Wang\footnote{The fourth author is partially supported by an NSF of China and the Department of Education of Shandong Province of China}}
\bigskip


\vspace{0.2cm}


\bigskip
\noindent\textbf{Abstract: } In the paper, we prove that reduced
free products of unital AH algebras with respect to given faithful
tracial states, in the sense of Voiculescu, are Blackadar and
Kirhcberg's MF algebras. We also show that the reduced free products
of unital AH algebras with respect to given  faithful tracial
states, under mild conditions, are not quasidiagonal. Therefore  we
conclude, for a large class of AH algebras, the
Brown-Douglas-Fillmore extension semigroups of  the reduced free
products of these AH algebras with respect to given faithful tracial
states are not groups.

 Our
result is based on Haagerup and Thorbj{\o}rsen's work on the reduced
C$^*$-algebras of free groups.

\vspace{0.2cm} \noindent{\bf Keywords:} MF algebras, Reduced free
products, BDF semigroups

\vspace{0.2cm} \noindent{\bf 2000 Mathematics Subject
Classification:} Primary 46L10, Secondary 46L54

 \vspace{0.2cm}

\section{Introduction}
The BDF theory was developed by    Brown,   Douglas and   Fillmore
in 1977 in \cite{BDF}.  In order to classify essentially normal
operators, they introduced an important  invariant, $Ext(\mathcal
A)$ (the BDF extension semigroup), for a unital separable
C$^*$-algebra $\mathcal A$. Among many other things they proved  in
\cite{BDF}   that $Ext(C(X))$ is a group when $X$ is a compact
metric space. Later, Choi and Effros \cite{CE}  showed that
$Ext(\mathcal A)$ is a group if $\mathcal A$ is a unital separable
nuclear C$^*$-algebra. By a result of Voiculescu, we know that the
semigroup  $Ext(\mathcal A)$ always has a unit if $\mathcal A$ is a
unital separable C$^*$-algebra.

 Anderson   \cite{An}   provided  the first example of a unital
separable C$^*$-algebra $\mathcal A$ such that $Ext(\mathcal A)$ is
not a group. Using Kazhdan's property $T$ for groups, Wassermann
gave other examples of unital separable C$^*$-algebras $\mathcal A$
such that $Ext(\mathcal A)$ is not a group in \cite{Wa}. In
\cite{Kir}, Kirchberg provided more examples of unital separable
C$^*$-algebras whose BDF extension semigroups are not groups by
showing that the following result:  A C$^*$-algebra $\mathcal A$ has
the local lifting property if and only if $Ext(S(\mathcal A))$ is a
group, where $S(\mathcal A)$ denotes the unitization of $C_0(\Bbb
R)\otimes_{min}\mathcal  A$.

Ever since Anderson's example in \cite{An}, it has been an open
problem whether $Ext(C_r^*(F_2))$, the BDF extension semigroup of
the reduced C$^*$-algebra of free group $F_2$, is a group. This
problem was studied by many mathematician (see \cite{V4}) and
finally settled down in the negative  by Haagerup and
Thorbj{\o}rnsen \cite{Haag} using powerful tools developed from
Voiculescu's free probability theory and random matrix theory. Their
result that $Ext(C_r^*(F_2))$ is not a group follows from a
combination of Voiculescu's result in \cite{V4} and their striking
work on showing that $C_r^*(F_2)$  can be   embedded into $\prod_k
\mathcal M_{n_k}(\Bbb C)/\sum_k \mathcal M_{n_k}(\Bbb C)$ for a
sequence of positive integers $\{n_k\}_{k=1}^\infty$.

If a separable C$^*$-algebra $\mathcal A$ can be embedded into
$\prod_k \mathcal M_{n_k}(\Bbb C)/\sum_k \mathcal M_{n_k}(\Bbb C) $
for a sequence of positive integers $\{n_k\}_{k=1}^\infty$, then
such C$^*$-algebra $\mathcal A$ is call an MF algebra. The concept
of MF algebra was introduced by Blackadar and Kirchberg in \cite{BK}
in order to study the classification problem of C$^*$-algebras. Many
properties of MF algebras were discussed there. Using the concept of
MF algebras, Brown in \cite{Br} (see also \cite{Haag}) generalized
Voiculescu's result in \cite{V4} as follows: {\em If a unital
separable C$^*$-algebra $\mathcal A$ is an MF algebra but not a
quasidiagonal C$^*$-algebra, then $Ext(\mathcal A)$ is not a group.
} Note by a result of Rosenberg, $C_r^*(F_2)$ is not quasidiagonal.
Now Haagerup and Thorbj{\o}rnsen's work can be restated as follows:
{\em $C_r^*(F_2)$ is an MF algebra and $Ext(C_r^*(F_2))$ is not a
group.}

The concept of reduced free products of unital C$^*$-algebras with
respect to given  states was provided by Voiculescu in the context
of his free probability theory \cite{Voi2}. This concept plays an
important role in the recent study of C$^*$-algebras (for example
see \cite{Dyk}, \cite{Dyk2}, \cite{DHR}). Assume that $(\mathcal
A,\tau_{\mathcal A})$  and $(\mathcal B,\tau_{\mathcal B})$  are
unital C$^*$-algebras with   faithful tracial states $\tau_{\mathcal
A}$, and $\tau_{\mathcal B}$ respectively. In \cite {Voi2},
Voiculescu introduced   the reduced free product $(\mathcal
A,\tau_{\mathcal A})*_{red}(\mathcal B,\tau_{\mathcal B})$  of
$(\mathcal A,\tau_{\mathcal A})$  and $(\mathcal B,\tau_{\mathcal
B})$. A quick fact  from the definition of reduced free product of
C$^*$-algebras is the following statement: {\em
$$
(C_r^*(F_2),\tau_{F_2})=(C_r^*(\Bbb Z ),\tau_{\Bbb Z  })*_{red}
(C_r^*( \Bbb Z ),\tau_{ \Bbb Z }),
$$ where, for a  discrete countable group $G$, we let $C_r^*(G )$
be the reduced C$^*$-algebra of group $G$ and $\tau_G$ the canonical
tracial state induced by the left regular representation $\lambda$,
i.e. $\tau_G(\lambda(g))=\langle \lambda(g)\delta_e,\delta_e\rangle$
for any $g$ in $G$ with $\delta_e$ the distinguished vector in
$l^2(G)$.}

In view of Haagerup and Thorbj{\o}rnsen's work on $C_r^*(F_2)$ and
the preceding  fact from the definition of reduced free products,
one should naturally consider the following question:
\begin{enumerate}
\item [] {\em What are   necessary and sufficient conditions on unital separable
 C$^*$-algebras $\mathcal A$ and $\mathcal B$   such
that $$Ext \big ((\mathcal A,\tau_{\mathcal A})*_{red}(\mathcal
B,\tau_{\mathcal B}) \big )\quad \text{ is not a group,}$$ where
$\tau_{\mathcal A}$, and $\tau_{\mathcal B}$, are faithful tracial
states of $\mathcal A$, and $\mathcal B$ respectively?}

\end{enumerate}
This paper grows out in an   attempt to understand Haagerup and
Thorbj{\o}rnsen's result in \cite{Haag} and search for answer to the
preceding question. In fact, we are able to prove the following
generalizations of Haagerup and Thorbj{\o}rnsen's result  mentioned
as above.

\vspace{0.2cm}

\textbf{Theorem 4.2.1: } {\em  Suppose that $\mathcal A_1$ and
$\mathcal A_2$   are unital separable  AH algebras with faithful
tracial states $\tau_1$, and $\tau_2$ respectively. If $\mathcal
A_1$ and $\mathcal A_2$ satisfy Avitzour's condition, {\em i.e.
there are unitaries $u\in \mathcal A_1$ and $v,w\in\mathcal A_2$
such that $ \tau_1(u)=\tau_2(v)=\tau_2(w)=\tau_2(w^*v)=0, $ } then}
$$
Ext \big ((\mathcal A_1,\tau_1)*_{red}(\mathcal A_2,\tau_2) \big
)\quad \text{\em is not a group.}
$$

Recall a unital separable C$^*$-algebra $\mathcal A$ is an
approximately homogeneous (AH) C$^*$-algebra if $\mathcal A$ is an
inductive limit of a sequence of   homogeneous C$^*$-algebras (see
\cite{Bl}). Obviously, all AF algebras, AI algebras and  A$\Bbb T$
algebras are AH algebras.

\vspace{0.2cm}

  \textbf{Theorem 4.2.2: } {\em  Let $\mathcal A$ and
$\mathcal B\ne \Bbb C$ be separable unital AH algebras with faithful
traces $\phi$, and $\psi$ respectively. If $\mathcal A$ is partially
diffuse in the sense of Definition 4.1.1, then}
$$
Ext \big ((\mathcal A,\phi)*_{red}(\mathcal B,\psi) \big ) \quad
\text{\em is not a group.}
$$

  \textbf{Theorem 4.2.3: } {\em  Suppose that $\mathcal
A $  and $\mathcal B$  are unital separable AF algebras with
faithful tracial states $\phi$, and $\psi$ respectively. If
$dim_{\Bbb C}\mathcal A\ge 2$ and $dim_{\Bbb C}\mathcal B\ge 3$,
then}
$$
Ext \big ((\mathcal A,\phi)*_{red}(\mathcal B,\psi) \big )\quad
\text{\em is not a group.}
$$

Using these results, one can easily produce new examples of unital
separable C$^*$-algebras whose BDF extension semigroups are not
groups. For example, {\em Let $\mathcal A$ and $\mathcal B $ be
irrational C$^*$-algebras, or UHF algebras, with faithful traces
$\phi$, and $\psi$ respectively.
 Then
$ Ext \big ((\mathcal A,\phi)*_{red}(\mathcal B,\psi) \big ) $  is
not a group.}  Combining with results from \cite{HaSh4} and
\cite{HaSh5}, one obtains more examples of unital separable
C$^*$-algebras whose BDF extension semigroups are not groups.

One  crucial step in proving Theorem 4.2.1, Theorem 4.2.2 and
Theorem 4.2.3 is our following result on Blackadar and Kirchberg's
MF algebra, whose proof is   based on Haagerup and Thorbj{\o}rnsen's
result that $C_r^*(F_2)$ is an MF algebra.

\vspace{0.2cm} \textbf{Theorem 3.3.3: } {\em  Suppose that $\mathcal
A_i$, $i=1,\ldots, n$, is a family of  unital separable AH algebras
with faithful tracial states $\tau_i$, $i=1,\ldots, n$. Then
$$
(\mathcal A_1,\tau_1)*_{red}\cdots *_{red}(\mathcal A_n,\tau_n)
$$ is an MF algebra.}

Blackadar and Kirchberg's MF algebra is also closely connected to
Voiculescu's topological free entropy dimension. In \cite{Voi}, for
a family of self-adjoint elements $x_1,\ldots, x_n$ in a unital
C$^*$-algebra $\mathcal A$, Voiculescu introduced the notion of
topological free entropy dimension of $x_1,\ldots,x_n$.  In the
definition of
  topological free entropy dimension, it requires  that the
C$^*$-subalgebra generated by $x_1,\ldots, x_n$ in $\mathcal A$ is
an MF algebra. Applying Theorem 3.3.3, we obtain the following
result on Voiculescu's topological free entropy dimension:

\vspace{0.2cm} \textbf{Corollary 3.3.1: }{\em  Suppose that
$(\mathcal A,\tau)$ is a C$^*$-free probability space. Let
$x_1,\ldots, x_n$ be a family of self-adjoint elements in $\mathcal
A$ such that $x_1,\ldots, x_n$ are  free with respect to $\tau$.
Then
$$
\delta_{top}(x_1,\ldots, x_n)\ge 0,
$$ where $
\delta_{top}(x_1,\ldots, x_n)     $ is the Voiculescu's topological
free entropy dimension.}

\vspace{0.2cm}

The organization of the paper is as follows. In section 2, we
introduce some notation and basic concepts needed in the later
sections. In section 3, we start with Haagerup and Thorbj{\o}rnsen's
result on $C_r^*(F_n)$ and show that the reduced free products of
finite dimensional C$^*$-algebras with respect to given tracial
states are MF algebras. Then we conclude that the reduced free
products of AH  algebras with respect to given tracial states are MF
algebras. In section 4, we show that the reduced free products of
unital  C$^*$-algebras, under mild conditions, are
non-quasidiagonal. Combining the results from section 3, we reach
our conclusions on reduced free products of unital AH algebras whose
BDF extension semigroups are not groups in section 4. In section 5,
we further discuss the reduced free products of some tensor products
of unital C$^*$-algebras, which are not covered in section 3.
\section{Notation and Preliminaries}

In this section, we will recall some basic facts on C$^*$-algebras
and introduce some lemmas that will be needed in the later sections.

\subsection{Gram-Schmidt orthogonalization} Suppose $\mathcal H$ is a complex Hilbert space
and $\langle \cdot, \cdot\rangle$ is an inner product on $\mathcal
H$. Let $\{y_m\}_{m=1}^N$ be a family of linearly independent
vectors in $\mathcal H$, where $N$ is a positive integer. Let, for
each $1\le m\le N$,
$$
P_m(y_1,\ldots, y_N) =   \left |
  \begin{aligned}
  &\langle y_1, y_1\rangle  & \quad \langle y_2, y_1\rangle  & \quad \cdots
  &  \quad \langle y_m, y_1\rangle \\
 &\langle y_1, y_2\rangle  & \quad \langle y_2, y_2\rangle  & \quad \cdots
  &  \quad \langle y_m, y_2 \rangle \\
   &  & & \quad  \cdots   &\\
    &\langle y_1, y_{m-1}\rangle  & \quad \langle y_2, y_{m-1}\rangle  & \quad \cdots
  &  \quad \langle y_m, y_{m-1}\rangle \\
   & \ \ \ y_1 \ \    & \quad  y_2 \ \ \   & \quad \cdots\ \ \
  &  \quad y_m \ \ \
  \end{aligned}
  \right |.
$$Since $y_1,\ldots, y_N$ are linearly independent, each
$P_m(y_1,\ldots,y_N)\ne 0$. We let
$$
p_m(y_1,\ldots, y_N)= \frac {P_m(y_1,\ldots, y_N)}{\langle
P_m(y_1,\ldots, y_N),P_m(y_1,\ldots, y_N) \rangle^{1/2}}, \qquad
for \ 1\le m\le N.
$$

Then we have the following statement.
\begin{lemma}
Assume $y_1,\ldots, y_N$ is a family of linearly independent
vectors in $\mathcal H$ and $p_m(y_1,\ldots, y_N)$ for $1\le m\le
N$ is defined  as above. Then $\{  p_m(y_1,\ldots, y_N)\}_{m=1}^N$
forms an orthonormal basis for the closed subspace spanned by
$y_1,\ldots, y_N$ in $\mathcal H$.
\end{lemma}

\subsection{Reduced crossed products of C$^*$-algebras by  groups}

Assume that $\mathcal A$ is a separable unital C$^*$-algebra and $G$
is a discrete countable group. Let $\alpha$ be a homomorphism from
$G$ into $Aut(\mathcal A)$. Then we can define the reduced crossed
product, $\mathcal A\rtimes_{\alpha,r} G$, of $\mathcal A$ by the
action   $\alpha$ of $G$ as follows. Let $\rho: \mathcal A
\rightarrow B(\mathcal H)$ be a faithful $*$-representation of
$\mathcal A$ on a separable Hilbert space $\mathcal H$. Let $l^2(G)$
be the Hilbert space associated to $G$ with an orthonormal basis
$\{e_g\}_{g\in G}$. Let $\lambda :G\rightarrow B(l^2(G))$ be the
left regular representation of $G$ on the Hilbert space $l^2(G)$.
Then we let $\mathcal K=\mathcal H\otimes l^2(G)$. And we introduce
a representation $\sigma$ of $\mathcal A$ and $G$ on $\mathcal K$ by
the following:
$$
  \begin{aligned}
      \sigma(g)&= 1\otimes \lambda(g), \qquad \qquad  \forall \ g\in G\\
      \sigma(x)(\xi\otimes e_g)&= (\alpha^{-1}(g)(x)\xi)\otimes e_g,
      \quad \forall \ x\in\mathcal A, \ \forall \ \xi\in\mathcal H, \ \ g\in G.
  \end{aligned}
$$
Then the C$^*$-algebra generated by $\{\sigma(g)\}_{g\in G}$ and
$\{\sigma(x)\}_{x\in\mathcal A}$ in $B(\mathcal K)$ is called the
reduced crossed product of $\mathcal A$ by $G$, and is denoted by
$\mathcal A\rtimes_{\alpha,r} G$. We know that the reduced crossed
product $\mathcal A\rtimes_{\alpha,r} G$ does not depend on the
choice of the faithful $*$-representation $\rho$ of $\mathcal A$.

\subsection{Reduced free products of unital C$^*$-algebras} The
concept of reduced free products of unital C$^*$-algebras was
introduced by Voiculescu in the  context of his free probability
theory. (see \cite{Voi2}, also \cite{Avit})

Assume that $\mathcal A_i$, $i=1,2$, is a separable unital
C$^*$-algebra with a   state $\tau_i$. For each $i=1,2$, let
$(\mathcal H_i,\xi_i,\pi_i)$ be the GNS representation of $\mathcal
A_i$ on the Hilbert space $\mathcal H_i$ such that (i)
$\tau_i(x_i)=\langle \pi_i(x_i)\xi_i,\xi_i\rangle$ for all $x_i\in
\mathcal A_i$ and (ii) $\mathcal H_i= \overline{\{\pi_i(x_i)\xi_i\ |
\ \ x_i\in\mathcal A_i\}}$.

Let $\stackrel{\circ}{\mathcal H_i}=\mathcal H_i\ominus \Bbb C\xi_i$
for $i=1,2$. The Hilbert space free product of $(\mathcal
H_1,\xi_1)$ and  $(\mathcal H_2,\xi_2)$ is given by
$$
\mathcal H=(\mathcal H_1,\xi_1)* (\mathcal H_2,\xi_2) =\Bbb C \xi
\oplus \bigoplus_{n\ge 1} \left ( \bigoplus_{j_1\ne j_2\ne \cdots
\ne j_n} \stackrel{\circ}{\mathcal H}_{j_1}\otimes \cdots \otimes
\stackrel{\circ}{\mathcal H}_{j_n} \right ),
$$ where $\xi$ is the distinguished unit vector in $\mathcal H$.
Let, for $i=1,2$,
$$
\mathcal H(i)  =\Bbb C \xi \oplus \bigoplus_{n\ge 1} \left (
\bigoplus_{i\ne j_1\ne j_2\ne \cdots \ne j_n}
\stackrel{\circ}{\mathcal H}_{j_1}\otimes \cdots \otimes
\stackrel{\circ}{\mathcal H}_{j_n} \right ).
$$
We can define  unitary operators $V_i:  {\mathcal H}_i\otimes
\mathcal H(i) \rightarrow \mathcal H$ as follows:
$$
\begin{aligned}
   \xi_i\otimes \xi &\mapsto \xi\\
   \stackrel{\circ}{\mathcal H}_i\otimes \xi &\mapsto  \stackrel{\circ}{\mathcal H}_i\\
   \xi_i\otimes (\stackrel{\circ}{\mathcal
H}_{j_1}\otimes \cdots \otimes \stackrel{\circ}{\mathcal H}_{j_n}) &
\mapsto \stackrel{\circ}{\mathcal H}_{j_1}\otimes \cdots \otimes
\stackrel{\circ}{\mathcal
H}_{j_n}\\
\stackrel{\circ}{\mathcal H}_i\otimes (\stackrel{\circ}{\mathcal
H}_{j_1}\otimes \cdots \otimes \stackrel{\circ}{\mathcal H}_{j_n})
&\mapsto \stackrel{\circ}{\mathcal H}_i\otimes
\stackrel{\circ}{\mathcal H}_{j_1}\otimes \cdots \otimes
\stackrel{\circ}{\mathcal H}_{j_n}
\end{aligned}
$$
Let $\lambda_i$ be the representation of $\mathcal A_i$ on
$\mathcal H$ given by
$$
\lambda_i(x) = V_i(\pi_i(x)\otimes I_{\mathcal H(i)})V_i^*, \qquad
\forall \ x\in\mathcal A_i.
$$

Then the reduced free product of $(\mathcal A_1,\tau_1)$ and
$(\mathcal A_2,\tau_2)$, or the reduced free product of $\mathcal
A_1$ and $\mathcal A_2$ with respect to $\tau_1$ and $\tau_2$, is
the C$^*$-algebra generated by $
 \lambda_1(\mathcal A_1) $ and $\lambda_2(\mathcal A_2)$ in
 $B(\mathcal H)$, and is denoted by
 $$
 (\mathcal A_1,\tau_1)*_{red}  (\mathcal A_2,\tau_2).
 $$
Moreover, the free product  state $\tau=\tau_1*\tau_2$ on $
 (\mathcal A_1,\tau_1)*_{red}  (\mathcal A_2,\tau_2),
 $ given by $\tau(x)=\langle x\xi,\xi\rangle$, is a faithful
 tracial state if both $\tau_1$ and $\tau_2$ are faithful tracial
 states on $\mathcal A_1$, and $\mathcal A_2$ respectively.

\begin{remark}
Suppose that $\mathcal A_1$, and $\mathcal A_2$ are unital
C$^*$-algebras with  faithful tracial states $\tau_1$, and $\tau_2$
respectively. Suppose that $I_{\mathcal A_1}\in\mathcal B_1$, and
$I_{\mathcal A_2}\in\mathcal B_2$, are  unital C$^*$-subalgebras of
$\mathcal A_1$, and $\mathcal A_2$ respectively. Then there is an
embedding
$$
(\mathcal B_1,\tau_1|_{\mathcal B_1})*_{red} (\mathcal
B_2,\tau_2|_{\mathcal B_2}) \subseteq (\mathcal A_1,\tau_1 )*_{red}
(\mathcal A_2,\tau_2).
$$
\end{remark}

\subsection{Blackadar and Kirchberg's MF algebras} Recall the
definition of Blackadar and Kirchberg's MF algebras (\cite{BK}) as
follows.
\begin{definition}
A separable C$^*$-algebra $\mathcal A$ is called an MF algebra if
there is an embedding from $\mathcal A$ into $\prod_{k=1}^\infty
\mathcal M_{n_k}(\Bbb C)/ \sum_{k=1}^\infty \mathcal M_{n_k}(\Bbb
C)$ for a sequence of positive integers $\{n_k\}_{k=1}^\infty$
where $\mathcal M_{n_k}(\Bbb C)$ is the $n_k\times n_k$ complex
matrix algebra.
\end{definition}
We will need the following lemma in the later sections.
\begin{lemma}
Let $\mathcal A$ be a unital separable C$^*$-algebra. Then the
following are equivalent.
\begin{enumerate}
  \item [(i)] $\mathcal A$ is an MF algebra;
  \item [(ii)] For any family of self-adjoint elements $x_1,\ldots, x_n$ in
  $\mathcal A$, any $\epsilon>0$, and any family of noncommutative
  polynomials $P_1,\ldots, P_r$ in $\Bbb C\langle X_1,\ldots,
  X_n \rangle$,   there are a positive integer
  $k$ and a family of self-adjoint matrices
  $$
A_1,\ldots, A_n \qquad in \ \ \mathcal M_k(\Bbb C)
  $$ such that
  $$
   \max_{1\le j\le r} \left |\|P_j(x_1,\ldots,x_n)\|-\|P_j(A_1,\ldots,A_n)\|  \right
   |\le \epsilon.
  $$
\end{enumerate}
\end{lemma}
\begin{proof} Note a separable C$^*$-algebra $\mathcal A$ is an MF algebra if and
only if every finitely generated C$^*$-subalgebra of $\mathcal A$ is
an MF algebra (see Corollary 3.4.4 in \cite{BK}). The rest follows
from Theorem 5.2 in \cite{HaSh4}.
\end{proof}

\begin{remark}A separable C$^*$-subalgebra of an  MF algebra is also an MF
algebra.
\end{remark}

\section{Reduced Free Products of AH Algebras}
In this section, we are going to show that   reduced free products
of unital AH algebras with respect to given faithful tracial states
 are MF algebras. First, we need to consider GNS
representation of a finite dimensional C$^*$-algebra.

\subsection{GNS representation of a finite dimensional
C$^*$-algebra} Suppose that $\mathcal B$ is a finite dimensional
C$^*$-algebra and $\psi$ is a faithful tracial state of $\mathcal
B$.

Let $d=dim_{\Bbb C} \mathcal B$, the complex dimension of $\mathcal
B$. Then there is a family of elements $ 1, b_1,\ldots, b_{d-1} $ in
$\mathcal B$ that forms a basis of $\mathcal B$, where $1$ is the
identity of $\mathcal B$. Note $\psi$ is a faithful tracial state of
$\mathcal B$. We can introduce an inner product on $\mathcal B $ as
follows.
$$
 \langle x,y\rangle =\psi(y^*x), \qquad \forall \ x,y\in\mathcal
 B.
$$ By the Gram-Schmidt orthogonalization   in Section 2.1,
for the basis $ 1, b_1,\ldots, b_{d-1} $ of $\mathcal B$, we let
$$P_1(1,b_1, \ldots, b_{d-1}:\psi)=1$$ and $$ P_m(1,b_1, \ldots,
b_{d-1}:\psi) =   \left |
  \begin{aligned}
  & \quad 1   & \quad \psi(b_1)    & \quad \cdots
  &  \quad   \psi(b_{m-1})  \\
 &\psi(b_1^*) & \quad \psi(b_1^*b_1)  & \quad \cdots
  &  \quad \psi(b_1^*b_{m-1})   \\
   &  & & \quad  \cdots   &\\
    &\psi(b_{m-2}^* )   & \quad \psi(b_{m-2}^*b_1)   & \quad \cdots
  &  \quad \psi(b_{m-2}^*b_{m-1}) \\
   & \ \ \ 1 \ \    & \quad b_1 \ \ \   & \quad \cdots\ \ \
  &  \quad b_{m-1} \ \ \
  \end{aligned}
  \right |, \qquad 2\le m\le d;
$$
and $$p_{m}(1,b_1, \ldots, b_{d-1}:\psi)=\frac {P_m(1,b_1, \ldots,
b_{d-1}:\psi)}{\left (\psi(P_m(1,b_1, \ldots,
b_{d-1}:\psi)^*P_m(1,b_1, \ldots,
b_{d-1}:\psi))\right)^{1/2}},\qquad 1\le m\le d.$$ Then
$$
1=p_1(1,b_1, \ldots, b_{d-1}:\psi), \  p_2(1,b_1, \ldots,
b_{d-1}:\psi),\ \ldots, \  p_d(1,b_1, \ldots, b_{d-1}:\psi)
$$ forms an orthonormal basis of $\mathcal B=L^2(\mathcal B,\psi)$.
\begin{lemma}
Suppose that $\mathcal B$ is a finite dimensional C$^*$-algebra with
a basis $1,b_1,\ldots, b_{d-1}$, where $d$ is the complex dimension
of $\mathcal B$. Suppose that $\psi$ is a faithful tracial state of
$\mathcal B$. Let
$$
1=p_1(1,b_1, \ldots, b_{d-1}:\psi), \  p_2(1,b_1, \ldots,
b_{d-1}:\psi),\ \ldots, \  p_d(1,b_1, \ldots, b_{d-1}:\psi)
$$ be defined as above.

Let $\Bbb C^d$ be a $d$-dimensional complex Hilbert space with an
orthonormal basis $e_1,\ldots, e_d$. Then there is a faithful unital
$*$-representation $\rho_{\psi}:\mathcal B\rightarrow \mathcal
M_d(\Bbb C)$ of $\mathcal B$ on $\Bbb C^d$ such that
\begin{enumerate}
 \item [(i)] $(\rho_{\psi}, \Bbb C^d, e_1)$ is a GNS representation
 of $(\mathcal B, \psi)$, i.e.
   \begin{enumerate}
     \item [(a)] $\psi(a) = \langle \rho_{\psi} (a)
     e_1,e_1\rangle$ for all $a\in\mathcal B$.
     \item [(b)] $\Bbb C^d= \overline {\{\rho_{\psi}(a)e_1 \ | \ a\in\mathcal
     B\}}$
   \end{enumerate}
   \item [(ii)] For each $1\le i\le d-1$,
   $$
     \rho_{\psi}(b_i) = B_{i,\psi}=\left [ b(s,t:i,\psi) \right
     ]_{s,t=1}^d \in \mathcal M_d(\Bbb C)
   $$ where $b(s,t:i,\psi)$, the $(s, t)$-th entry of the matrix $B_{i,\psi}$,
   is given by
   $$
b(s,t:i,\psi)= \psi(p_t(1,b_1,\ldots, b_{d-1}:\psi)^* b_i
p_s(1,b_1,\ldots, b_{d-1}:\psi))
   $$
\end{enumerate}
\end{lemma}
\begin{proof}
Note that $\mathcal B$ is a finite dimensional C$^*$-algebra with a
faithful tracial state $\psi$. We can  view $\mathcal B=L^2(\mathcal
B,\psi)$ as a Hilbert space with the inner product induced from
$\psi$. Thus $\mathcal B=L^2(\mathcal B,\psi)$ is isomorphic to
$\Bbb C^d$ as a Hilbert space. By the explanation preceding the
lemma, we can introduce a unitary $U: L^2(\mathcal
B,\psi)\rightarrow \Bbb C^d$ by mapping
$$
p_m(1,b_1,\ldots, b_{d-1}:\psi)\mapsto e_m, \qquad \forall \ 1\le
m\le d.
$$
Apparently, such   $U$ induces a faithful unital $*$-representation
$\rho_{\psi}: \mathcal B\rightarrow \mathcal M_d(\Bbb C)$ by
$$
\rho_{\psi}(b)= UbU^*, \qquad \forall \ b\in \mathcal B.
$$ Now it is easy to verify that $(\rho_{\psi}, \Bbb C^d, e_1)$ is a GNS representation
 of $(\mathcal B, \psi)$ satisfying (a) and
 (b). Moreover, for each $1\le i\le d-1$,
   $$
     \rho_{\psi}(b_i) = B_{i,\psi}=\left [ b(s,t:i,\psi) \right
     ]_{s,t=1}^d \in \mathcal M_d(\Bbb C)
   $$ satisfying
   $$
b(s,t:i,\psi)= \psi(p_t(1,b_1,\ldots, b_{d-1}:\psi)^* b_i
p_s(1,b_1,\ldots, b_{d-1}:\psi)).
   $$
This completes the proof.
\end{proof}

\begin{lemma}
Suppose that $\mathcal B$ is a finite dimensional C$^*$-algebra with
a basis $1,b_1,\ldots, b_{d-1}$, where $d$ is the complex dimension
of $\mathcal B$. Suppose that
$\{\tau,\tau_\gamma\}_{\gamma=1}^\infty$ is a family of faithful
tracial states of $\mathcal B$ satisfying $$ \lim_{\gamma\rightarrow
\infty}\tau_{\gamma}(b)=\tau(b) \qquad \forall \ b\in\mathcal B .$$

Let $\Bbb C^d$ be a $d$-dimensional complex Hilbert space with an
orthonormal basis $e_1,\ldots, e_d$. Then there is a sequence of
faithful unital $*$-representations $\rho_{\tau}, \rho_{\tau_\gamma}
:\mathcal B\rightarrow \mathcal M_d(\Bbb C)$ of $\mathcal B$ on
$\Bbb C^d$ for $\gamma=1,2\ldots$ such that
\begin{enumerate}
 \item [(i)] $(\rho_{\tau}, \Bbb C^d, e_1)$  and  $(\rho_{\tau_\gamma}, \Bbb C^d, e_1)$  are
  GNS representations
 of $(\mathcal B, \tau)$, and $ (\mathcal B,  \tau_\gamma)$
 respectively.
   \item [(ii)] For each $1\le i\le d-1$,
   $$
\lim_{\gamma\rightarrow \infty}
\|\rho_{\tau_\gamma}(b_i)-\rho_{\tau}(b_i)\|=0
   $$
   \end{enumerate}
\end{lemma}
\begin{proof}
Note  $\mathcal B$ is a finite dimensional C$^*$-algebras with a
basis $1,b_1,\ldots, b_{d-1}$ and
$\{\tau,\tau_\gamma\}_{\gamma=1}^\infty$ is a family of faithful
tracial states of $\mathcal B$.  $\Bbb C^d$ is a $d$-dimensional
complex Hilbert space with an orthonormal basis $e_1,\ldots, e_d$.
By Lemma 3.1.1, there is a sequence of  faithful unital
$*$-representations $\rho_{\tau}, \rho_{\tau_\gamma} :\mathcal
B\rightarrow \mathcal M_d(\Bbb C)$ of $\mathcal B$ on $\Bbb C^d$ for
$\gamma=1,2\ldots$ such that \begin{enumerate}
 \item [(iii)] $(\rho_{\tau}, \Bbb C^d, e_1)$  and  $(\rho_{\tau_\gamma}, \Bbb C^d, e_1)$  are
  GNS representations
 of $(\mathcal B, \tau)$, and $ (\mathcal B,  \tau_\gamma)$
 respectively.
  \item [(iv)] Moreover, for each $1\le i\le d-1$,
   $$
     \rho_{\tau}(b_i) = B_{i,\tau}=\left [ b(s,t:i,\tau) \right
     ]_{s,t=1}^d \in \mathcal M_d(\Bbb C)
   $$ satisfying
   $$
b(s,t:i,\tau)= \tau(p_t(1,b_1,\ldots, b_{d-1}:\tau)^* b_i
p_s(1,b_1,\ldots, b_{d-1}:\tau));
   $$ and, for $\gamma=1,2,\ldots$
$$
     \rho_{\tau_\gamma}(b_i) = B_{i,\tau_\gamma}=\left [ b(s,t:i,\tau_\gamma) \right
     ]_{s,t=1}^d \in \mathcal M_d(\Bbb C)
   $$ satisfying
   $$
b(s,t:i,\tau_\gamma)= \tau_\gamma(p_t(1,b_1,\ldots,
b_{d-1}:\tau_\gamma)^* b_i p_s(1,b_1,\ldots, b_{d-1}:\tau_\gamma)).
   $$
 \end{enumerate}
Since
$$
\lim_{\gamma\rightarrow \infty} \tau_{\gamma}(b)=\tau(b),  \qquad
\forall \ b\in \mathcal B,
$$
by the choices of $$1=p_1(1,b_1,  \ldots,  b_{d-1}:\tau), \ \
\ldots, \ \ p_d(1,b_1,\ldots, b_{d-1}:\tau)$$ and
$$1=p_1(1,b_1,\ldots, b_{d-1}:\tau_\gamma), \ \ \ldots, \ \ p_d(1,b_1,\ldots,
b_{d-1}:\tau_\gamma)$$ in the discussion before Lemma 3.1.1, we know
that
$$
\lim_{\gamma\rightarrow \infty} b(s,t:i,\tau_\gamma)= b(s,t:i,\tau).
$$ It follows that
$$
\lim_{\gamma\rightarrow \infty}
\|\rho_{\tau_\gamma}(b_i)-\rho_{\tau}(b_i)\|\le
\lim_{\gamma\rightarrow \infty} d^2 \left (\max_{1\le s,t\le d} |
b(s,t:i,\tau_\gamma)- b(s,t:i,\tau)  |\right ) =0, \quad \forall \
1\le i \le d-1.
$$
\end{proof}

\begin{definition}
Suppose that $\mathcal A$ and $\mathcal B$ are separable unital
C$^*$-algebras. Let $\epsilon>0$ be a positive number. Suppose that
$x_1,\ldots, x_n$ is a family of elements in $\mathcal A$. Then we
call
$$ \{x_1,\ldots, x_n\}\subseteq_\epsilon \mathcal B
$$ if the following holds:
\begin{enumerate}
  \item [] There are
 (i) $y_1,\ldots, y_n$ in $\mathcal B$ and (ii)  unital faithful
  $^*$-representations $\rho_1: \mathcal A\rightarrow B(\mathcal H)$ and
  $\rho_2:\mathcal B\rightarrow B(\mathcal H)$ on a Hilbert space
  $\mathcal H$ such that
  $$
\max_{1\le i\le n} \|\rho_1(x_i)-\rho_2(y_i)\| \le \epsilon.
  $$
\end{enumerate}
\end{definition}

\begin{lemma}
     Suppose that $\mathcal A$ is a separable unital C$^*$-algebra
     with a faithful tracial state $\psi$. Suppose that $\mathcal B$
     is a finite dimensional C$^*$-algebra with a family
     $\{\tau,\tau_{\gamma}\}_{\gamma=1}^\infty$ of faithful tracial
     states of $\mathcal B$ such that
     $$
           \lim_{\gamma\rightarrow
           \infty}\tau_{\gamma}(b)=\tau(b),\qquad \forall \
           b\in\mathcal B.
     $$

     Suppose that $x_1,\ldots, x_n$ is a family of elements in $ (\mathcal A,\psi)*_{red}(\mathcal
     B,\tau)$.
     Then, for any $\epsilon>0$, there is a $\gamma_0>0$ such that
     $$
        \{ x_1,\ldots, x_n\}
         \subseteq_\epsilon (\mathcal A,\psi)*_{red}(\mathcal
         B,\tau_\gamma), \qquad \forall \ \gamma>\gamma_0.
     $$
\end{lemma}
\begin{proof} Note that $\mathcal B$
     is a finite dimensional C$^*$-algebra. Assume that $1, b_1,\ldots, b_{d-1}$ is a basis of $\mathcal B$, where
      $d$ is the complex dimension of $\mathcal B$. Let $\Bbb C^d$ be a $d$-dimensional complex Hilbert space with
an orthonormal basis $e_1,\ldots, e_d$. By Lemma 3.1.2,   there is a
sequence of faithful unital $*$-representations $\rho_{\tau},
\rho_{\tau_\gamma} :\mathcal B\rightarrow \mathcal M_d(\Bbb C)$ of
$\mathcal B$ on $\Bbb C^d$ for $\gamma=1,2\ldots$ such that
\begin{enumerate}
 \item [(i)] $(\rho_{\tau}, \Bbb C^d, e_1)$ and  $(\rho_{\tau_\gamma}, \Bbb C^d, e_1)$  are
  GNS representations
 of $(\mathcal B, \tau)$, and $ (\mathcal B,  \tau_\gamma)$
 respectively.
   \item [(ii)] For each $1\le i\le d-1$,
   \begin{equation}
\lim_{\gamma\rightarrow \infty}
\|\rho_{\tau_\gamma}(b_i)-\rho_{\tau}(b_i)\|=0
   \end{equation}
   \end{enumerate}

Let $(\pi, \mathcal H_1, \xi_1)$ be the GNS representation of
$(\mathcal A,\psi)$ on a Hilbert space $\mathcal H_1$ such that
$\xi_1$ is cyclic for $\pi(\mathcal A)$ and $\psi(a)=\langle
\pi(a)\xi_1,\xi_1\rangle$ for all $a$ in $\mathcal A$.

Let $(\mathcal H_2,\xi_2)=(\Bbb C^d, e_1)$ and $\mathcal H$ be the
free product of Hilbert spaces $(\mathcal H_1,\xi_1)$ and $(\mathcal
H_2,\xi_2)$ as in Section 2.3. Let $\stackrel{\circ}{\mathcal H}_i$
and $\mathcal H(i)$ be defined as in Section 2.3 for $i=1,2$ and
$V_1,V_2$ be the unitary operators as defined in Section 2.3. Let
$\lambda$ be the representation of $\mathcal A$ and $\mathcal B$ on
the Hilbert space $\mathcal H$ defined as follows:
   \begin{align}
    \lambda(a) &= V_1(\pi(a)\otimes I_{\mathcal H(1)})V_1^*, \qquad
    \forall \ a\in \mathcal A;\\
    \lambda (b) &= V_2(\rho_\tau(b)\otimes I_{\mathcal H(2)})V_2^*,
    \qquad \forall \ b\in\mathcal B;
   \end{align}
 Let $\lambda_\gamma$, $\gamma=1,2,\ldots,$ be a sequence of
representations of $\mathcal A$ and $\mathcal B$ on the Hilbert
space $\mathcal H$ defined as follows:
   \begin{align}
    \lambda_\gamma(a) &= V_1(\pi(a)\otimes I_{\mathcal H(1)})V_1^*, \qquad
    \forall \ a\in \mathcal A;\\
    \lambda_\gamma (b) &= V_2(\rho_\gamma(b)\otimes I_{\mathcal H(2)})V_2^*,
    \qquad \forall \ b\in\mathcal B;
   \end{align}
Then by the definition of reduced free product   in Section 2.3, we
know that \begin{enumerate}\item[(a)] $(\mathcal
A,\psi)*_{red}(\mathcal
     B,\tau)$ is the unital C$^*$-subalgebra of $B(\mathcal H)$ generated by
     $$
    \{ \lambda(a) \ | \ a\in\mathcal A\}
    \cup  \{\lambda (b) \ | \ b\in\mathcal
    B\};
     $$
\item [(b)]  $(\mathcal A,\psi)*_{red}(\mathcal
     B,\tau_\gamma)$ is the unital C$^*$-subalgebra of $B(\mathcal H)$ generated by
    \begin{equation*}
    \{ \lambda_\gamma(a)  \ | \ a\in\mathcal A\}
    \cup  \{\lambda_\gamma (b)  \ | \ b\in\mathcal
    B\}.
   \end{equation*} \end{enumerate}
Moreover, by (3.1.2) and (3.1.4), we know that
\begin{equation}
 \lambda (a)= \lambda_\gamma(a) ,\qquad \forall \ a\in \mathcal A
\end{equation}
 By (3.1.1), (3.1.3) and (3.1.5), we know that, for
$1\le i\le d-1$,
 \begin{equation}
\lim_{\gamma\rightarrow \infty} \|\lambda(b_i)-\lambda_\gamma(b_i)\|
= \lim_{\gamma\rightarrow \infty} \|V_2(\rho_\gamma(b_i)\otimes
 I_{\mathcal H(2))}V_2^*-V_2(\rho_\tau (b_i)\otimes  I_{\mathcal H(2))}V_2^*\|=0.
 \end{equation}

Since $x_1,\ldots,x_n$ are in $(\mathcal A,\psi)*_{red}(\mathcal
     B,\tau)$, there are elements $a_1,\ldots, a_N$ in $\mathcal A$ and
     noncommutative polynomials $P_1,\ldots, P_n$ such that, for
     all $1\le j\le n$,
     $$\begin{aligned}
   &\|x_j-P_j(\lambda(a_1), \ldots, \lambda(a_N),
\lambda(b_1),\ldots, \lambda(b_{d-1})
   )\|  \le \epsilon/3. \end{aligned}
     $$
On the other hand, by (3.1.6) and (3.1.7), we know  when $\gamma$ is
large enough, for all $ 1\le j\le n$.
$$
  \begin{aligned}
   &\| P_j(\lambda(a_1), \ldots, \lambda(a_N),
\lambda(b_1),\ldots, \lambda(b_{d-1})
   ) - P_j(\lambda_\gamma(a_1), \ldots, \lambda_\gamma(a_N),
\lambda_\gamma(b_1),\ldots, \lambda_\gamma(b_{d-1})
   ) \|  \le \epsilon /3.
  \end{aligned}
$$
Therefore, when $\gamma$ is large enough,
    $$  \begin{aligned}
  & \|x_j-P_j(\lambda_\gamma(a_1), \ldots, \lambda_\gamma(a_N),
\lambda_\gamma(b_1),\ldots, \lambda_\gamma(b_{d-1})
   )\| \le \epsilon.\end{aligned}
     $$ I.e.  when $\gamma$ is large enough, we have
$
        \{ x_1,\ldots, x_n\}
         \subseteq_\epsilon (\mathcal A,\psi)*_{red}(\mathcal
         B,\tau_\gamma) .
     $\end{proof}
\subsection{Reduced free products of matrix algebras}
In this subsection, we will show that the reduced free products of
matrix algebras with respect to given tracial states are  MF
algebras. Recall the following remarkable result of Haagerup and
Thorbj{\o}rnsen.
\begin{lemma}
  [Haagerup and Thorbj{\o}rnsen] For all positive integer $n\ge 2$,
  $C_r^*(F_n)$ is an MF algebra, where $F_n$ is the nonabelian free group on $n$ generators and $C_r^*(F_n)$ is the
  reduced group C$^*$-algebra of the free group $F_n$.
\end{lemma}
The following result can be found in Thereom 4.1 in \cite{HaSh5}.
\begin{lemma}
Suppose $\mathcal A$ is a unital MF algebra and $G$ is a finite
group. Suppose that $\alpha: G\rightarrow Aut(\mathcal A)$ is a
homomorphism from $G$ into $Aut(\mathcal A)$. Then the reduced
crossed product, $\mathcal A\rtimes _{\alpha,r} G$, of $\mathcal A$
by $G$ is an MF algebra.
\end{lemma}
\begin{proof}
For the purpose of completeness, we sketch its proof here. In fact
the proof follows directly from the definition of reduced crossed
product in section 2.2.

Recall the definition of reduced crossed product as follows. Assume
that $\mathcal A$ acts on a Hilbert space $\mathcal H$. Let $l^2(G)$
be the Hilbert space associated to the group $G$ with an orthonormal
basis $\{e_{g}\}_{g\in G}$ and $\lambda$ be the left regular
representation of $G$ on $l^2(G)$. Let $E_g$ be the rank one
projection from $l^2(G)$ onto the vector $e_g$ in $l^2(G)$. Let
$\sigma$ be the representation of $\mathcal A$ and $G$ on $\mathcal
H\otimes l^2(G)$ induced by the following mapping:
$$\begin{aligned}
\sigma(g)& = I_{\mathcal H}\otimes \lambda (g),\qquad \forall \ g\in
G;\\
\sigma (x) &=\sum_{g\in G} \alpha_g^{-1}(x)\otimes E_g, \qquad
\forall \ x\in\mathcal A. \end{aligned}
$$ Then the reduced crossed product,
$\mathcal A\rtimes_{\alpha,r} G$, of $\mathcal A$ by $G$ is the
C$^*$-subalgebra generated by $\{\sigma(g)\}_{g\in G}\cup
\{\sigma(x)\}_{x\in\mathcal A}$ in $B(\mathcal H\otimes l^2(G))$.

Note $G$ is a finite group. Then $B(l^2(G))\simeq \mathcal M_k(\Bbb
C)$ for some positive integer $k$. Moreover, for all $g\in G$ and
$x\in \mathcal A$, $\sigma(g)$ and $\sigma(x)$ are in $\mathcal
A\otimes B(l^2(G))$. Since $\mathcal A$ is an MF algebra, $\mathcal
A\otimes B(l^2(G))$ is also an MF algebra (see Proposition 3.3.6 in
\cite{BK}). It follows that $\mathcal A\rtimes_{\alpha,r} G$, as a
C$^*$-subalgebra of $\mathcal A\otimes B(l^2(G))$, is also an MF
algebra. This completes the proof of the lemma.
\end{proof}
A quick corollary of the preceding lemma is the following statement.
\begin{corollary}
For any positive integer $n\ge 2$, $C_r^*(\Bbb Z_n*F_n)$ is an MF
algebra, where $\Bbb Z_n$ is the quotient group $\Bbb Z/n\Bbb Z$
and $\Bbb Z_n*F_n$ is the free product of group $\Bbb Z_n$ and the
free group $F_n$.
\end{corollary}
\begin{proof}
Assume that $u$ is a natural generator of $\Bbb Z_n$, i.e. $u^n=e$
and $u^j\ne e$ for all $1\le j<n$. Assume that $g_1,\ldots, g_n$ are
the natural generators of $F_n$. Let $\alpha$ be an action of $\Bbb
Z_n$ on $F_n$ induced by the following mapping:
$$
\alpha(u)(g_i)=g_{i+1} \quad for \ 1\le i\le n-1, \qquad and \ \ \
\alpha(u)(g_n)=g_1.
$$Let $h_i=g_1^ig_2^i\ldots
g_n^i$ for $i=1,2,\ldots, n$ be elements in $F_n$. Then we observe
that as elements in $ F_n\rtimes_{\alpha} \Bbb Z_n$, the elements
$u$, $h_1$,\ldots, $h_n$ are free in $ F_n\rtimes_{\alpha} \Bbb
Z_n$. In other words, $\Bbb Z_n*F_n$ can be viewed as a subgroup of
$ F_n\rtimes_{\alpha} \Bbb Z_n$. Therefore,
$$
C_r^*(\Bbb Z_n*F_n)\subseteq C_r^*(F_n\rtimes_{\alpha} \Bbb
Z_n)=C_r^*(F_n)\rtimes_{\alpha,r} \Bbb Z_n.
$$
By Haagerup and Thorbj{\o}rnsen's result and Lemma 3.2.2, we know
that $ C_r^*(\Bbb Z_n*F_n)$ is an MF algebra.
\end{proof}

\begin{lemma}
For any $n\ge 2$, let $\tau_n$ be the normalized trace on
$\mathcal M_n(\Bbb C)$. Then
$$
(\mathcal M_n(\Bbb C),\tau_n)*_{red} (\mathcal M_n(\Bbb C),\tau_n)
$$ is an MF algebra.
\end{lemma}
\begin{proof}
Assume that the group $\Bbb Z_n*F_n$ is generated by the natural
generators $u$ in $\Bbb Z$ and $g_1,\ldots, g_n$ in $F_n$. Let
$\lambda$ be the left regular representation of $\Bbb Z_n*F_n$ on
the Hilbert space $l^2( \Bbb Z_n*F_n)$ with the cyclic and
separating vector $\eta_1$. Thus $C_r^*(\Bbb Z_n*F_n)$ is generated
by $\lambda(u)$ and $\lambda(g_1),\ldots, \lambda(g_n)$ in
$B(l^2(\Bbb Z_n*F_n))$.

 Assume the
second copy of $\Bbb Z_n$ is generated by another natural generator
$v$. Let $\gamma$ in $\Bbb C$ be the $n$-th root of unit. Let
$\alpha: \Bbb Z_n\rightarrow Aut(C^*_r(\Bbb Z_n*F_n))$ be a
homomorphism from $\Bbb Z_n$ into $Aut(C^*_r(\Bbb Z_n*F_n))$ induced
by the following mapping:
$$
\begin{aligned}
     \alpha(v)(\lambda(u))&= \gamma \lambda(u) \\
     \alpha(v)(\lambda(g_i))&=\lambda(g_{i+1}) , \ \ \ for \ 1\le
     i\le n-1\\
      \alpha(v)(\lambda(g_n))&=\lambda(g_{ 1}).
\end{aligned}
$$

Let $C^*_r(\Bbb Z_n*F_n)\rtimes_{\alpha,r} \Bbb Z_n$ be the reduced
crossed product of $C^*_r(\Bbb Z_n*F_n)$ by the group $\Bbb Z_n$.
Recall the definition of reduced crossed product of C$^*$-algebras
as in Section 2.2. Let $l^2(\Bbb Z_n)$ be the Hilbert space with an
orthonormal basis $e_1, e_v, e_{v^2},\ldots, e_{v^{n-1}}$. Let
$\lambda: \Bbb Z_n\rightarrow B(l^2(\Bbb Z_n))$ be the left regular
representation of $\Bbb Z_n$ on the Hilbert space $l^2(\Bbb Z_n)$
with the cyclic and separating vector $e_1$. Let $\mathcal H=l^2(
\Bbb Z_n*F_n)\otimes l^2(\Bbb Z_n)$. Then we introduce
representation $\sigma$ of $\Bbb Z_n$ and $ \Bbb Z_n*F_n $ on
$\mathcal H$ as following
$$
\begin{aligned}
   \sigma( g )&=I_{l^2( \Bbb Z_n*F_n)}\otimes \lambda(g), \qquad
   \forall \
   \ g\in \Bbb Z_n\\
   \sigma(h) (\xi\otimes e_{v^i})&= (\alpha^{-1}(v^i)(\lambda(h))(\xi))\otimes
   e_{v^i},\qquad \forall \ h\in  \Bbb Z_n*F_n, \ \forall \ 1\le
   i\le n.
\end{aligned}
$$ Then  $C^*_r(\Bbb Z_n*F_n)\rtimes_{\alpha,r} \Bbb Z_n$ is the
C$^*$-algebra generated by $\{\sigma(g), \sigma(h) \ | \ g\in\Bbb
Z_n \ and \  h\in
  \Bbb Z_n*F_n \}$ in $B(\mathcal H)$. And we have
  $$
    \sigma(v)\sigma(u)=\gamma \sigma(u)\sigma(v).
  $$
Furthermore, there is a canonical faithful tracial state $\tau$ on
$C^*_r(\Bbb Z_n*F_n)\rtimes_{\alpha,r} \Bbb Z_n$, which is
defined by
$$
\tau(x)= \langle x (\eta_1\otimes e_1), \eta_1\otimes e_1\rangle,
\qquad \forall \ x\in C^*_r(\Bbb Z_n*F_n)\rtimes_{\alpha,r} \Bbb
Z_n.
$$

\begin{claim}
$\{\sigma( v ),\sigma( u )\}$ and $\{\sigma( g_1 )\}$ are free
with respect to $\tau$ in $C^*_r(\Bbb Z_n*F_n)\rtimes_{\alpha,r}
\Bbb Z_n$
\end{claim}

[Proof of Claim:] Note that $\{\sigma(u^i)\sigma(v^j)\ | \ 0\le
i,j\le n-1\}$ forms a basis for the C$^*$-subalgebra generated by
$\sigma(u)$ and $\sigma(v )$ in $C^*_r(\Bbb
Z_n*F_n)\rtimes_{\alpha,r} \Bbb Z_n$. And
$\{\sigma(g_1^t)\}_{t=-\infty}^\infty$ forms a basis for the
C$^*$-subalgebra generated by $\sigma(g_1)$ in $C^*_r(\Bbb
Z_n*F_n)\rtimes_{\alpha,r} \Bbb Z_n$. Therefore to prove the claim
it suffices to show the following: {\em For any positive integer
$r$, nonzero integers $n_1,  n_2, ,\ldots, n_r $, and integers
$m_1,k_1,\ldots, m_r,k_r$ with  $0\le m_i,k_i<n$ and  $(m_i,k_i)\ne
(0,0)$ for $1\le i\le r$, we have
$$
\tau(\sigma(g_1^{n_1})\sigma(u^{m_1})\sigma(v^{k_1})\cdots
\sigma(g_1^{n_r})\sigma(u^{m_r})\sigma(v^{k_r}))=0.
$$}
Note that
$$\begin{aligned}
&\tau(\sigma(g_1^{n_1})\sigma(u^{m_1})\sigma(v^{k_1})\cdots
\sigma(g_1^{n_r})\sigma(u^{m_r})\sigma(v^{k_r})) = \tau(\left
(\sigma(v^{k_r})\sigma(g_1^{n_1}) \sigma(v^{k_r})^*\right)  \left
(\sigma(v^{k_r})\sigma(u^{m_1}) \sigma(v^{k_r})^*\right) \\
&\qquad\quad \left (\sigma(v^{k_r+k_1})\sigma(g_1^{n_2})
\sigma(v^{k_r+k_1})^*\right) \cdots \left
(\sigma(v^{k_r+k_1+\cdots k_{r-1}})\sigma(u^{m_r})
\sigma(v^{k_r+k_1+\cdots
k_{r-1}})^*\right)\sigma(v^{k_r+k_1+\cdots k_{r-1} }) )
\end{aligned}
$$
 Thus it will be enough if we are able to show the
following is true:  {\em For any positive integer $r$, nonzero
integers $n_1, n_2, ,\ldots, n_r $, and integers $m_1,p_1,\ldots,
m_r,p_r$ with  $0\le m_i,p_i<n$ and   $(m_i,p_{i+1}-p_i)\ne
(0,0)$ for $1\le i\le r$, we have
$$
\tau(\left ( \sigma(v^{p_1})\sigma(g_1^{n_1})\sigma(v^{p_1})^*
\right ) \sigma(u^{m_1})\left (
\sigma(v^{p_2})\sigma(g_1^{n_2})\sigma(v^{p_2})^*\right )\cdots
\left ( \sigma(v^{p_{r }})\sigma(g_1^{n_r})\sigma(v^{p_{r
}^*})\right )\sigma(u^{m_r}))=0,
$$}
The last equality is equivalent to:
$$
\tau(  \sigma(g_{1+p_1}^{n_1})   \sigma(u^{m_1})
 \sigma(g_{1+p_2}^{n_2}) \cdots
 \sigma(g_{1+p_{r }}^{n_r}) \sigma(u^{m_r}))=0.
$$On the other hand, by the freeness of $u$ and $g_1,\ldots, g_n$
in $\Bbb Z_n*F_n$, we know, if $(m_i,p_{i+1}-p_i)\ne (0,0)$ for
$1\le i\le r$ then
$$
 g_{1+p_1}^{n_1} u^{m_1} g_{1+p_2}^{n_2}\cdots g_{1+p_{r }}^{n_r}
 u^{m_r} \quad \text{ is a reduced word in $\Bbb Z_n*F_n$}.
$$Therefore
$$
\tau(  \sigma(g_{1+p_1}^{n_1})   \sigma(u^{m_1})
 \sigma(g_{1+p_2}^{n_2}) \cdots
 \sigma(g_{1+p_{r }}^{n_r}) \sigma(u^{m_r}))= 0.
$$
This implies that $\{\sigma( v ),\sigma( u )\}$ and $\{\sigma( g_1
)\}$ are free with respect to $\tau$ in $C^*_r(\Bbb
Z_n*F_n)\rtimes_{\alpha,r} \Bbb Z_n$. This ends the proof of the
claim.

[Continue the proof of the lemma:] Since $u$ and $v$ are two natural
generators of the group $\Bbb Z_n$ and $ \sigma(u) \sigma(v) =\gamma
\sigma(v)\sigma(u) $  where  $\gamma$ is the $n$-th root of the
unit, the C$^*$-subalgebra $\mathcal B$  generated by $\sigma(u)$
and $\sigma(v)$ in $C^*_r(\Bbb Z_n*F_n)\rtimes_{\alpha,r} \Bbb Z_n$
is $*$-isomorphic to $\mathcal M_n(\Bbb C)$.    By Claim, we know
that $\mathcal B$ and $\sigma(g_1) \mathcal B\sigma(g_1)^*$ are free
with respect to $\tau$ in $C^*_r(\Bbb Z_n*F_n)\rtimes_{\alpha,r}
\Bbb Z_n$. Since $\tau$ is a faithful tracial state on $C^*_r(\Bbb
Z_n*F_n)\rtimes_{\alpha,r} \Bbb Z_n$, $\tau$ is also a faithful
tracial state on the C$^*$-subalgebra generated by $\mathcal B$ and
$\sigma(g_1)\mathcal B \sigma(g_1)^*$. Combining with the fact that
$\mathcal B$   is $*$-isomorphic to $\mathcal M_n(\Bbb C)$, we know
that
$$
(\mathcal M_n(\Bbb C),\tau_n)*_{red} (\mathcal M_n(\Bbb C),\tau_n)
\simeq(\mathcal B,\tau|_{\mathcal B})*_{red} (\mathcal
B,\tau|_{\mathcal B}) \simeq C^*(\mathcal B,\sigma(g_1)\mathcal B
\sigma(g_1)^*) \subseteq C^*_r(\Bbb Z_n*F_n)\rtimes_{\alpha,r} \Bbb
Z_n.
$$
By Lemma 3.2.2 and Corollary 3.2.1, we know that $$ (\mathcal
M_n(\Bbb C),\tau_n)*_{red} (\mathcal M_n(\Bbb C),\tau_n)$$ is an MF
algebra.
\end{proof}

\begin{definition}
  Suppose that $$\mathcal B\simeq \mathcal M_{n_1}(\Bbb C) \oplus
  \mathcal M_{n_2}(\Bbb C)\oplus \cdots \oplus \mathcal M_{n_r}(\Bbb
  C)$$ is a finite dimensional C$^*$-algebra. Let  $\tau_{n_i}$ be the normalized tracial state on
  $\mathcal M_{n_i}(\Bbb C)$ for each $1\le i\le r$. Moreover every element $x$
  in $\mathcal B$ can be written as
  $$
   x=x_1\oplus x_2\oplus \cdots \oplus x_r, \qquad with \ each \
   x_i\in \mathcal M_{n_i}(\Bbb C), \ \forall \ 1\le i\le r.
  $$
  Then   a  tracial state   $\tau$ on $\mathcal B$ is called a rational tracial state if there are
  rational numbers $0\le \alpha_1,\ldots, \alpha_r\le 1$ such that
  $$
    \tau(x) = \alpha_1\tau_{n_1}(x_1)+\cdots+
    \alpha_r\tau_{n_r}(x_r), \qquad \forall \ x\in \mathcal B.
  $$
\end{definition}

\begin{proposition}
Suppose that $\mathcal B_1$  and $\mathcal B_2$  are finite
dimensional C$^*$-algebras with   faithful rational tracial states
$\tau_1$, and $\tau_2$ respectively. Then
$$
(\mathcal B_1,\tau_1)*_{red} (\mathcal B_2,\tau_2)
$$ is an MF algebra.
\end{proposition}
\begin{proof}
Since both $\tau_1$ and $\tau_2$ are faithful rational tracial
states on $\mathcal B_1$, and $\mathcal B_2$ respectively, there are
a positive integer $n$ and    trace-preserving, faithful, unital
$*$-monomorphisms $\pi_1: \mathcal B_1\rightarrow \mathcal M_n(\Bbb
C),$ and $ \pi_2: \mathcal B_2\rightarrow \mathcal M_n(\Bbb C), $
such that
$$
\tau_n(\pi_1(x_1))=\tau_1(x_1) \qquad and \qquad
\tau_n(\pi_2(x_2))=\tau_2(x_2), \qquad \forall \ x_1\in \mathcal
B_1, x_2\in\mathcal B_2,
$$ where $\tau_n$ is the tracial state on $\mathcal M_n(\Bbb C)$.
Therefore,
$$
(\mathcal B_1,\tau_1)*_{red} (\mathcal B_2,\tau_2)\subseteq
(\mathcal M_n(\Bbb C),\tau_n)*_{red} (\mathcal M_n(\Bbb C),\tau_n).
$$ By Lemma 3.2.2, we know that $
(\mathcal B_1,\tau_1)*_{red} (\mathcal B_2,\tau_2)$ is an MF
algebra.
\end{proof}

\subsection{Reduced free products of unital AH algebras} Recall the
definition of unital AF algebra as follows.
\begin{definition}
A unital separable C$^*$-algebra $\mathcal A$ is called
approximately finite dimensional (AF) algebra if for every
$x_1,\ldots, x_n$ in $\mathcal A$ and every $\epsilon>0$, there is a
 finite dimensional C$^*$-subalgebra $I_{\mathcal A}\in \mathcal B\subseteq
\mathcal A$ satisfying
$$
\max_{1\le i\le n} dist(x_i,\mathcal B)\le \epsilon.
$$
\end{definition}

The following lemma is quite useful.
\begin{lemma}
 Suppose that $\mathcal A$ is a separable C$^*$-algebra. Assume for every
 $x_1,\ldots, x_n$ in $\mathcal A$ and every $\epsilon>0$, there is
 an MF algebra $\mathcal A_1$ such that
 $$
 \{x_1,\ldots, x_n\}\subseteq_{\epsilon} \mathcal A_1 \qquad  \text {(in the sense of Definition 3.3.1)}.
 $$ Then $\mathcal A$ is also an MF algebra.
\end{lemma}
\begin{proof}
Assume that $x_1,\ldots, x_n$ is a family of self-adjoint elements
in $\mathcal A$ and $\epsilon >0$. Suppose that $P_1,\ldots, P_r$ is
a family of noncommutative polynomials in $\Bbb C\langle X_1,\ldots,
X_n\rangle$. Therefore, by condition on $\mathcal A$, we know that
there is an MF algebra $\mathcal A_1$ and a family of self-adjoint
elements $y_1,\ldots, y_n$ in $\mathcal A_1$ such that, for all
$1\le j\le r$
$$
\left | \|P_j(x_1,\ldots, x_n)\|-\|P_j(y_1, \ldots, y_n)\| \right |
\le \epsilon/2.
$$ Since $\mathcal A_1$ is an MF algebra, there are a positive
integer $k$ and self-adjoint matrices
$$
A_1,\ldots, A_n\in \mathcal M_k^{s.a.}(\Bbb C)
$$ such that, for all $1\le j\le r$,
$$
\left | \|P_j(A_1,\ldots, A_n)\|-\|P_j(y_1, \ldots, y_n)\| \right |
\le \epsilon/2.
$$ It follows that,  for all $1\le j\le r$,
$$
\left | \|P_j(A_1,\ldots, A_n)\|-\|P_j(x_1, \ldots, x_n)\| \right |
\le \epsilon.
$$
By Lemma 2.4.1, we know that $\mathcal A$ is an MF algebra.
\end{proof}

\begin{lemma}
Suppose that $\mathcal B_1$  and $\mathcal B_2$   are finite
dimensional C$^*$-algebras with   faithful tracial states $\tau $,
and $\psi$ respectively. Then
$$
(\mathcal B_1,\tau )*_{red} (\mathcal B_2,\psi)
$$ is an MF algebra.
\end{lemma}
\begin{proof} Suppose that $x_1\ldots, x_n$ is a family of elements
in $ (\mathcal B_1,\tau )*_{red} (\mathcal B_2,\psi) $ and
$\epsilon>0$ is a positive number.

    Apparently, there is a sequence of faithful rational tracial
    states $\tau_{\alpha}$, $\alpha=1,2\ldots$, on $\mathcal B_1$ such that
    $$
      \lim_{\alpha\rightarrow \infty}\tau_{\alpha} (b ) =\tau(b ),
      \qquad \forall \ b \in\mathcal B_1.
    $$ Thus by Lemma 3.1.3, there is a positive integer $\alpha$
such that
$$
\{x_1,\ldots, x_n\}\subseteq_{\epsilon} (\mathcal B_1,\tau_\alpha
)*_{red} (\mathcal B_2,\psi),\ \ \text {(in the sense of Definition
3.3.1)}
$$ whence there are   $y_1,\ldots, y_n$ in $(\mathcal B_1,\tau_\alpha
)*_{red} (\mathcal B_2,\psi)$ and unital faithful
$*$-representations $\rho_1$ of $(\mathcal B_1,\tau  )*_{red}
(\mathcal B_2,\psi)$ and $\rho_2$ of $(\mathcal B_1,\tau_\alpha
)*_{red} (\mathcal B_2,\psi)$ on a Hilbert space $\mathcal H$ such
that
$$
\max_{1\le i\le n} \|\rho_1(x_i)-\rho_2(y_i)\|\le \epsilon.
$$

Applying Lemma 3.1.3 again, we know that there is a faithful
rational tracial state $\psi_\beta$ on $\mathcal B_2$ such that
$$
\{y_1,\ldots, y_n\}\subseteq_{\epsilon} (\mathcal B_1,\tau_\alpha
)*_{red} (\mathcal B_2,\psi_\beta).
$$
I.e. there are   $z_1,\ldots, z_n$ in $(\mathcal B_1,\tau_\alpha
)*_{red} (\mathcal B_2,\psi_\beta)$ and unital faithful
$*$-representations $\rho_3$ of $(\mathcal B_1,\tau_\alpha  )*_{red}
(\mathcal B_2,\psi)$ and $\rho_4$ of $(\mathcal B_1,\tau_\alpha
)*_{red} (\mathcal B_2,\psi_\beta)$ on a Hilbert space $\mathcal K$
such that
$$
\max_{1\le i\le n} \|\rho_3(y_i)-\rho_4(z_i)\|\le \epsilon.
$$
Without loss of generality, we can assume that both $\rho_2$ and
$\rho_3$ are unital, faithful, essential representations, i.e. there
is no nonzero compact operator in the ranges of $\rho_2$ and
$\rho_3$. By a result in \cite{Voi1}, there is a sequence of
unitaries $u_k : \mathcal H\rightarrow \mathcal K,$ for
$k=1,2,\ldots$, such that
$$
\limsup_{k\rightarrow \infty}\|\rho_2(y)-u_k^*\rho_3(y)
u_k\|=0,\qquad \forall \ y\in(\mathcal B_1,\tau_\alpha  )*_{red}
(\mathcal B_2,\psi).
$$
It follows that, $\forall \ 1\le i\le n$,
$$\begin{aligned}
\limsup_{k\rightarrow \infty}&\|\rho_1(x_i)- u_k^*\rho_4(z_i)
u_k\|\\&\le \limsup_{k\rightarrow \infty}\big(
\|\rho_1(x_i)-\rho_2(y_i)\|+\|\rho_2(y_i)-u_k^*\rho_3(y_i)
u_k\|+\|u_k^*\rho_3(y_i) u_k- u_k^*\rho_4(z_i) u_k\|\big)\\&\le
2\epsilon.\end{aligned}
$$

 Altogether, we have that
$$
\{x_1,\ldots, x_n\}\subseteq_{3\epsilon} (\mathcal B_1,\tau_\alpha
)*_{red} (\mathcal B_2,\psi_\beta).
$$By Proposition 3.2.1, we know that $(\mathcal B_1,\tau_\alpha
)*_{red} (\mathcal B_2,\psi_\beta)$ is an MF algebra. Thus by Lemma
3.3.1, we know that $ (\mathcal B_1,\tau )*_{red} (\mathcal
B_2,\psi) $ is an MF algebra.
\end{proof}

\begin{theorem}
Suppose that $\mathcal A_1$  and $\mathcal A_2$  are unital
separable AF subalgebras with faithful tracial states $\tau_1$, and
$\tau_2$ respectively. Then
$$
(\mathcal A_1,\tau_1)*_{red}(\mathcal A_2,\tau_2)
$$ is an MF algebra.
\end{theorem}
\begin{proof}
Suppose that $x_1\ldots, x_n$ is a family of elements in $ (\mathcal
A_1,\tau_1 )*_{red} (\mathcal A_2,\tau_2) $ and $\epsilon>0$ is a
positive number.

By the definition of AF algebra, we know that there are  finite
dimensional C$^*$-algebras $I_{\mathcal A_i}\in \mathcal
B_i\subseteq \mathcal A_i$ for $i=1,2$ such that
$$
\{x_1,\ldots, x_n\}\subseteq_{\epsilon} \text{the C$^*$-subalgebra
generated by $\mathcal B_1$ and $\mathcal B_2$ in $(\mathcal
A_1,\tau_1 )*_{red} (\mathcal A_2,\tau_2)$ }.
$$
Since $\tau_1*\tau_2$ is a faithful tracial state of $(\mathcal
A_1,\tau_1 )*_{red} (\mathcal A_2,\tau_2)$ and $I_{\mathcal A_1}\in
\mathcal B_1$, $I_{\mathcal A_2}\in\mathcal B_2$, we know that the
C$^*$-subalgebra generated by $\mathcal B_1$ and $\mathcal B_2$ in
$(\mathcal A_1,\tau_1 )*_{red} (\mathcal A_2,\tau_2)$  is
$*$-isomorphic to
$$
 (\mathcal B_1,\tau_1|_{\mathcal B_1} )*_{red} (\mathcal B_2,\tau_2|_{\mathcal
 B_2}).
$$ Therefore,
$$
\{x_1,\ldots, x_n\}\subseteq_{\epsilon}(\mathcal
B_1,\tau_1|_{\mathcal B_1} )*_{red} (\mathcal B_2,\tau_2|_{\mathcal
 B_2}).
$$ By Lemma 3.3.2, we know that $ (\mathcal
B_1,\tau_1|_{\mathcal B_1} )*_{red} (\mathcal B_2,\tau_2|_{\mathcal
 B_2})
$ is an MF algebra. It follows from Lemma 3.3.1 that $$ (\mathcal
A_1,\tau_1)*_{red}(\mathcal A_2,\tau_2)
$$ is an MF algebra.
\end{proof}
\begin{lemma}
  The following statements are true:
  \begin{enumerate}
    \item Suppose that $X$ is a compact metric space and $C(X)$ is the unital C$^*$-algebra
   consisting all continuous functions on $X$. Let $\tau$ be
    a faithful tracial state on $C(X)$. Then there are a unital separable  AF
    algebra $\mathcal A$ with a faithful trace $\psi$ and a unital embedding
    $\rho:C(X) \rightarrow \mathcal A$ such that
    $\tau(x)=\psi(\rho(x))$ for all $x\in C(X)$.
    \item Suppose that $\mathcal B\simeq \oplus_{i=1}^k
    \big(\mathcal M_{n_i}(\Bbb C)\otimes C(X_i)\big)$ is a unital separable
    C$^*$-algebra with a faithful tracial state $\tau$, where each
    $X_i$ is a compact metric space for $1\le i\le k$. Then there are a unital separable  AF
    algebra $\mathcal A$ with a faithful trace $\psi$ and a unital embedding
    $\rho:\mathcal B \rightarrow \mathcal A$ such that
    $\tau(x)=\psi(\rho(x))$ for all $x\in \mathcal B$.
  \end{enumerate}
\end{lemma}
\begin{proof}
    It suffices to prove (1). Since $X$ is a compact metric space,
    $C(X)$ is a separable C$^*$-algebra. We might assume that
    $\{x_n\}_{n=1}^\infty$ is a dense subset in $C(X)$. Let $\rho$
    be the GNS representation of $C(X)$ on the Hilbert
    space $L^2(C(X),\tau)$. Any element $a$ in $C(X)$ corresponds to a vector $\hat a$
    in $L^2(C(X),\tau)$.
     Let $\psi$ be the vector state defined by $\psi(T)=\langle
     T\hat 1,\hat 1\rangle$ for all $T$ in $B(L^2(C(X),\tau))$, where $1$ is the unit of $C(X)$. Let
     $\mathcal M$ be the von Neumann algebra generated by
     $\rho(C(X))$ in $B(L^2(C(X),\tau))$.
     Since $\tau$ is a faithful trace of $C(X)$ and $\rho$ is the GNS representation of the unital C$^*$-algebra
     $C(X)$,  we know $\rho$ is
    a     faithful $*$-representation of $C(X)$ and $\mathcal M$ is an abelian von Neumann algebra with a
    faithful tracial state $\psi$.

    By spectral theory, for each $x_n$ in $C(X)$, there is a sequence of projections
    $\{p_{n,k}\}_{k=1}^\infty$ in $\mathcal M$ such that $\rho(a)$
    is in the C$^*$-subalgebra, $C^*(\{p_{n,k}\}_{k=1}^\infty)$,
    generated by $\{p_{n,k}\}_{k=1}^\infty$ in $\mathcal M$. Let
    $\mathcal A$ be the unital C$^*$-algebra generated by
    $\{p_{n,k}\}_{n,k=1}^\infty$ in $\mathcal M$. Therefore
     $\mathcal A$ is an unital AF algebra.
    Moreover $\psi$ is a
    faithful tracial state on $\mathcal A$ and $\rho$ is a unital
    embedding from $C(X)$ into $\mathcal A$ satisfying
    $\tau(x)=\psi(\rho(x))$ for all $x\in C(X)$.
\end{proof}

Recall the definition of AH algebra in the sense of Blackadar (see
Definition 2.1 in \cite{Bl}).
\begin{definition}
A unital separable C$^*$-algebra $\mathcal A$ is an approximately
homogeneous (AH) C$^*$-algebra if $\mathcal A$ is an inductive limit
of a sequence of   homogeneous C$^*$-algebras $\mathcal A_m,
m=1,2,\ldots$, where each $\mathcal
A_m=\oplus_{i=1}^{k_m}\big(\mathcal M_{[m,n_i]}(\Bbb C)\otimes
C(X_{[m,i]})\big)$ and each $X_{[m,i]}$ is a compact matric space.
\end{definition}

\begin{theorem}
Suppose that $\mathcal A_1$  and $\mathcal A_2$  are unital
separable AH algebras with faithful tracial states $\tau_1$, and
$\tau_2$ respectively. Then
$$
(\mathcal A_1,\tau_1)*_{red}(\mathcal A_2,\tau_2)
$$ is an MF algebra.
\end{theorem}
\begin{proof}
 For each $i=1,2$, the unital AH algebra $\mathcal A_i$  is an inductive limit of   homogeneous subalgebras
$\{\mathcal A_m^{(i)}\}_{m=1}^\infty$, each of which is
$*$-isomorphic to some $\oplus_{j=1}^k (\mathcal M_{n_j}(\Bbb
C)\otimes C(X_j))$ where each
    $X_j$ is a compact metric space for $1\le j\le k$. By Lemma
    3.3.3, we know for every $x_1,\ldots, x_n$ in $
(\mathcal A_1,\tau_1)*_{red}(\mathcal A_2,\tau_2) $ and
$\epsilon>0$,
    there are a positive integer $m$, unital AF algebras $\mathcal D_1$  and $\mathcal
    D_2$   with faithful tracial states $\psi_1$, and
    $\psi_2$ respectively, such that
    $$
     \{x_1,\ldots, x_n\}\subset_{\epsilon} (\mathcal
     \mathcal A_m^{(1)},\tau_1)*_{red} (\mathcal
     \mathcal A_m^{(2)},\tau_2)  \subseteq  (\mathcal
     D_1,\psi_1)*_{red} (\mathcal
     D_2,\psi_2).
    $$ By Lemma 3.3.1 and Theorem 3.3.1, we know that $
(\mathcal A_1,\tau_1)*_{red}(\mathcal A_2,\tau_2) $  is an MF
algebra.
\end{proof}

Recall a C$^*$-algebra $\mathcal A$ is called a {\em local
AH-algebra} if for every $\epsilon > 0$ and for every finite subset
$a_1, \ldots, a_n$ of $\mathcal A$ there is a  C$^*$-subalgebra
$\mathcal B$ of $\mathcal A$ which   (i) is   homogeneous, i.e.
$\mathcal B$ is $*$-isomorphic to a C*-algebra of the form
$\oplus_{i=1}^{k_m}\big(\mathcal M_{[m,n_i]}(\Bbb C)\otimes
C(X_{[m,i]})\big)$, and (ii) contains elements $b_1,\ldots, b_n$
with $\|a_j-b_j\|\le \epsilon$ for $j = 1,\ldots,n.$ Similar
argument in the proof of Theorem 3.3.2 proves the following
statement.
\begin{proposition}
Suppose that $\mathcal A_1$  and $\mathcal A_2$  are unital,
separable, local AH algebras with faithful tracial states $\tau_1$,
and $\tau_2$ respectively. Then
$$
(\mathcal A_1,\tau_1)*_{red}(\mathcal A_2,\tau_2) \qquad \text{ is
an MF algebra.} $$
\end{proposition}

Apparently, Theorem 3.3.2 can be generalized as follows.

\begin{theorem}
Suppose that $\mathcal A_i$, $i=1,\ldots, n$, is a family of  unital
separable AH algebras   with  faithful tracial states $\tau_i$. Then
$$
(\mathcal A_1,\tau_1)*_{red}\cdots *_{red}(\mathcal A_n,\tau_n)
\qquad \text{ is an MF algebra.} $$
\end{theorem}
\begin{proof}
Let $\mathcal A= \mathcal A_1\otimes_{min}\cdots
\otimes_{min}\mathcal A_n$ be the minimal tensor product of
$\mathcal A_1,\ldots, \mathcal A_n$ and $\tau=\tau_1\otimes_{min}
\cdots \otimes_{min} \tau_n$ be the tensor product of the  tracial
states $\tau_1,\ldots,\tau_n$. Then $\mathcal A$ is an AH algebra
with a faithful trace $\tau$. Let $C_r^*(\Bbb Z)$ be the reduced
C$^*$-algebra of the group $\Bbb Z$ with a canonical faithful trace
$\tau_{\Bbb Z} $.

By Proposition 3.3.1 or Theorem 3.3.2, we know that $(\mathcal
A,\tau)*_{red} (C_r^*(\Bbb Z), \tau_{\Bbb Z})$ is an MF algebra. And
we note that
$$
(\mathcal A_1,\tau_1)*_{red}\cdots *_{red}(\mathcal
A_n,\tau_n)\subseteq (\mathcal A ,\tau )*_{red}\cdots
*_{red}(\mathcal A ,\tau )\subseteq(\mathcal A,\tau)*_{red}
(C_r^*(\Bbb Z), \tau_{\Bbb Z}).
$$ Therefore $$
(\mathcal A_1,\tau_1)*_{red}\cdots *_{red}(\mathcal A_n,\tau_n)
$$ is an MF algebra.
\end{proof}

\begin{corollary}
Suppose that $(\mathcal A,\tau)$ is a C$^*$-free probability space.
Let $x_1,\ldots, x_n$ be a family of self-adjoint elements in
$\mathcal A$ such that $x_1,\ldots, x_n$ are free with respect to
$\tau$. Then the C$^*$-subalgebra generated by $x_1,\ldots, x_n$ in
$\mathcal A$ is an MF algebra. In particular,
$$
\delta_{top}(x_1,\ldots, x_n)\ge 0,
$$ where $
\delta_{top}(x_1,\ldots, x_n)   $ is the Voiculescu's topological
free entropy dimension.
\end{corollary}
\begin{proof}
It follows directly from the definition of Voiculescu's topological
free entropy dimension \cite{Voi} and the definition of MF algebra
(see Lemma 2.4.1).
\end{proof}
More discussions on topological free entropy dimension can be found
in \cite{HaSh2}, \cite{HaSh3}, and \cite{HaSh4}.

\subsection{Reduced free product of unital ASH algebras}
Recall if $\mathcal B$ is a unital sub-homogeneous C$^*$-algebra (or
equivalently, a C$^*$-subalgebra of a homogeneous C$^*$-algebra),
then all irreducible $*$-representations of $\mathcal B$ are finite
dimensional with $dim\le n$ for some positive integer $n\in\Bbb  N$.
Suppose $\tau$ is a faithful tracial state of $\mathcal B$ and
$\mathcal H$ is the Hilbert space  $ L^2(\mathcal B,\tau)  $. Each
element $a$ in $\mathcal B$ corresponds to a vector $\hat a$ in
$L^2(\mathcal B,\tau)$. Let $\pi$ be the GNS representation of
$\mathcal B$ on the Hilbert space $L^2(\mathcal B,\tau)$. Then the
von Neumann algebra, $\pi(\mathcal B)''$, generated by $\pi(\mathcal
B)$  in $B(L^2(\mathcal B,\tau))$ has the form:
$$
\pi(\mathcal B)'' \simeq \oplus_{k=1}^n ( \mathcal M_k(\Bbb C)
\otimes \mathcal B_k ),
$$where each $\mathcal B_k$ is either a unital abelian von Neumann algebra or $0$.  Let $\psi$ be the vector state defined by $\psi(T)=\langle
     T\hat 1,\hat 1\rangle$ for all $T$ in $B(L^2(\mathcal B,\tau))$,
      where $1$ is the unit of $\mathcal B$.
     Since $\tau$ is a faithful trace of $\mathcal B$ and $\pi$ is the GNS representation of the
     unital C$^*$-algebra $\mathcal B$, we know  $\psi$ is  a
faithful trace   of $\pi(\mathcal B)''$ and
$$
\tau(x)=\psi(\pi(x)), \qquad \forall \ x\in\mathcal B.
$$ Hence $\pi$ is a unital, trace-preserving, embedding from $\mathcal B$ into $\pi(\mathcal B)''$ such that
$$
(\mathcal B, \tau)\subseteq (\pi(\mathcal B)'', \psi).
$$
 By similar
argument  as in Lemma 3.3.3, we have the following result.
\begin{lemma}
Suppose that $\mathcal B$ is a unital separable sub-homogeneous
C$^*$-algebra with a faithful tracial state $\tau$. Then there  are
a unital separable AF
    algebra $\mathcal A$ with a faithful trace $\psi$ and a unital embedding
    $\rho:\mathcal B  \rightarrow \mathcal A$ such that
    $\tau(x)=\psi(\rho(x))$ for all $x\in \mathcal B$.
\end{lemma}

Recall an ASH  algebra (approximately sub-homogeneous C$^*$-algebra)
is an inductive limit of  a sequence of  sub-homogeneous
C$^*$-algebras. By similar arguments as in Theorem 3.3.2 and Theorem
3.3.3 (using Lemma 3.4.1 instead of Lemma 3.3.3), we have the
following result.
\begin{theorem}
Suppose that $\mathcal A_i$, $i=1,\ldots, n$, is a family of  unital
separable ASH algebras with  faithful tracial states $\tau_i$. Then
$$
(\mathcal A_1,\tau_1)*_{red}\cdots *_{red}(\mathcal A_n,\tau_n)
$$ is an MF algebra.
\end{theorem}

\section{BDF Extension Semigroups and Reduced Free Products of AH
Algebras}

Recall the definition of quasidiagonal C$^*$-algebra  as follows.
  \begin{definition} A set of elements $\{a_1,\ldots,a_n\}\subseteq B(\mathcal H)$ is quasidiagonal
if there  is an increasing sequence of finite-rank projections
$\{p_i\}_{i=1}^\infty$ on $\mathcal H$ tending strongly to the
identity such that $\| a_j p_i-p_i a_j \|\rightarrow 0$ as
$i\rightarrow \infty$ for any $1\le j\le n.$ A  separable
C$^*$-algebra
 $\mathcal A\subseteq B(\mathcal H)$ is
quasidiagonal    if there  is an increasing sequence of finite-rank
projections $\{p_i\}_{i=1}^\infty$ on $\mathcal H$ tending strongly
to the identity such that $\| x p_i-p_i x \|\rightarrow 0$
as $i\rightarrow \infty$ for any $ x \in   \mathcal A.$     
 An abstract separable
C$^*$-algebra $\mathcal A$ is quasidiagonal if there is a faithful
$*$-representation  $\pi: \mathcal A\rightarrow B(\mathcal H)$ such
that $\pi(\mathcal A)\subseteq B(\mathcal H)$ is quasidiagonal.
\end{definition}

By a result of Rosenberg, we know  that $C_r^*(F_2)$ is not
quasidiagonal.

We will    use the fact that a  C$^*$-subalgebra of a separable
quasidiagonal C$^*$-algebra is also quasidiagonal. In other words, a
separable C$^*$-algebra, containing a non-quasidiagonal
C$^*$-subalgebra, is
  not quasidiagonal.

\subsection{Non-quasidiagonality of reduced free products of AH
algebras} In this subsection, we are going to discuss
quasidiagonality of reduced free products of unital C$^*$-algebras.
Some of the conclusions stated in this subsection are direct
consequences of  results from other literature.

The following result might have been known to experts. For the
purpose of completeness, we include it here.
\begin{theorem}
Suppose that $\mathcal A_1$  and $\mathcal A_2$  are unital
separable C$^*$-algebras with faithful tracial states $\tau_1$, and
$\tau_2$ respectively. If $\mathcal A_1$ and $\mathcal A_2$ satisfy
Avitzour's condition, {\em  i.e. there are unitaries $u\in \mathcal
A_1$ and $v,w\in\mathcal A_2$ such that $$
\tau_1(u)=\tau_2(v)=\tau_2(w)=\tau_2(w^*v)=0,
$$}{then
$$
 (\mathcal A_1,\tau_1)*_{red}(\mathcal A_2,\tau_2)
$$ is not a quasidiagonal C$^*$-algebra.
}
\end{theorem}
\begin{proof}
   Let $a=uvuv$ and $b=uwuw$ be unitaries in $(\mathcal A_1,\tau_1)*_{red}(\mathcal
   A_2,\tau_2)$. Then we know that $a$ and $b$ are two Haar
   unitary elements in $(\mathcal A_1,\tau_1)*_{red}(\mathcal
   A_2,\tau_2)$ with respect
to the trace $\tau_1*\tau_2$. We note that
   $$
    \begin{aligned}
      &ab = uvuvuwuw \quad &ab^*=uvu(vw^*)u^*w^*u^* \quad
      &a^*b =v^*u^*(v^* w)uw \quad &a^*b^* =v^*u^*v^*u^*w^*
      u^*w^*u^*\\
      &ba  =uwuwuvuv \quad &ba^* = uwu(wv^*)u^*v^*u^*\quad &b^*a =
      w^*u^*(w^* v)uv \quad &b^*a^* = w^*u^*w^*u^*v^*u^*v^*u^*
    \end{aligned}
   $$
Now it is not hard to check that $a$ and $b$ are free with respect
to $\tau_1*\tau_2$. In other words, $C_r^*(F_2)$ is a
C$^*$-subalgebra of $(\mathcal A_1,\tau_1)*_{red}(\mathcal
A_2,\tau_2).$ Since $C_r^*(F_2)$ is not a quasidiagonal
C$^*$-algebra, $(\mathcal A_1,\tau_1)*_{red}(\mathcal A_2,\tau_2)$
is not a quasidiagonal C$^*$-algebra.
\end{proof}

The following result of N. Brown (see Corollary 4.3.6 in \cite{Br2})
is also useful in determining the quasidiagonality of a unital
C$^*$-algebra.
\begin{lemma} [Brown]
Suppose that $\mathcal A$ is a unital, separable, exact
C$^*$-algebra with a unique trace $\tau$. Let $\rho:\mathcal
A\rightarrow B(L^2(\mathcal A,\tau))$ be the GNS representation of
$\mathcal A$ on the Hilbert space $L^2(\mathcal A,\tau)$. If
$\mathcal A$ is quasidiagonal, then $\rho(\mathcal A)''$, the von
Neumann algebra generated by $\rho(\mathcal A)$ in $B(L^2(\mathcal
A,\tau))$, is a hyperfinite von Neumann algebra.
\end{lemma}

\begin{proposition}
Suppose that $C(\Bbb T)$ is the unital C$^*$-algebra consisting all
continuous functions on the unit circle $\Bbb T$ and  $\tau$ is a
faithful trace of $C(\Bbb T)$ induced by the Lesbeague measure on
$\Bbb T$. Suppose that $\mathcal B\ne \Bbb C$ is a unital,
separable,   C$^*$-algebra with a faithful tracial state $\psi$.
Then
$$
(C(\Bbb T),\tau)*_{red}(\mathcal B,\psi)
$$ is not a quasidiagonal C$^*$-algebra.
\end{proposition}
\begin{proof}We might assume that $\mathcal B$ is an exact
C$^*$-algebra. In fact, let $1\ne v$ be a unitary in $\mathcal B$
and $I_{\mathcal B}\in \mathcal B_1$ be a unital C$^*$-subalgebra of
$\mathcal B$ generated by $v$ in $\mathcal B$. Since $ (C(\Bbb
T),\tau)*_{red}(\mathcal B_1,\psi) $ is a C$^*$-subalgebra of $
(C(\Bbb T),\tau)*_{red}(\mathcal B,\psi) $, to show $ (C(\Bbb
T),\tau)*_{red}(\mathcal B,\psi) $ is not quasidiagonal  it suffices
to show that $ (C(\Bbb T),\tau)*_{red}(\mathcal B_1,\psi) $ is not
quasidiagonal. Apparently $\mathcal B_1\ne \Bbb C$ is a unital exact
C$^*$-algebra with a faithful trace $\psi$.

 By Dykema's result in Theorem 2 of  \cite{Dyk}, we
know that $ (C(\Bbb T),\tau)*_{red}(\mathcal B,\psi) $ is a simple
C$^*$-algebra with a unique tracial state $\tau*\psi$. Also by his
result in Theorem 3.5 of \cite{Dyk2}, we know that $ (C(\Bbb
T),\tau)*_{red}(\mathcal B,\psi) $ is an exact C$^*$-algebra.

Let $\rho$ be the GNS  representation of $ (C(\Bbb
T),\tau)*_{red}(\mathcal B,\psi) $ on the Hilbert space $L^2(
(C(\Bbb T),\tau)*_{red}(\mathcal B,\psi) ,\tau*\psi)$. Assume that $
(C(\Bbb T),\tau)*_{red}(\mathcal B,\psi) $ is a quasidiagonal
C$^*$-algebra. Then by Lemma 4.1.1, we know that $\rho( (C(\Bbb
T),\tau)*_{red}(\mathcal B,\psi))''$ is a hyperfinite von Neumann
algebra. Let $u$ be a Haar unitary in $C(\Bbb T)$ with respect to
$\tau$ and $v\ne 1$ be a  unitary in $\mathcal B$. Then by
Voiculescu's result in \cite{V2}, we know that
$$
\delta_0(\rho(u),\rho(v))=\delta_0(\rho(u))+\delta_0(\rho(v))=1+\delta_0(\rho(v))>1,
$$ where $\delta_0$ is the modified free entropy dimension for
finite von Neumann algebras. On the other hand, since  $\rho(
(C(\Bbb T),\tau)*_{red}(\mathcal B,\psi))''$ is a hyperfinite von
Neumann algebra, by \cite{V3} or \cite{HaSh}, we know
$$
\delta_0(\rho(u),\rho(v))\le 1.
$$ This is the
contradiction. Hence $\rho( (C(\Bbb T),\tau)*_{red}(\mathcal
B,\psi))''$ is not a hyperfinite von Neumann algebra. It follows
that $ (C(\Bbb T),\tau)*_{red}(\mathcal B,\psi) $  is not a
quasidiagonal C$^*$-algebra.
\end{proof}

The following useful result was obtained by Dykema in Proposition
2.8 of \cite{Dyk}.
\begin{lemma}[Dykema]
Let $\mathcal A = \mathcal A_1\oplus \mathcal A_2$ be a direct sum
of unital C$^*$-algebras. Write $p = 1\oplus 0 \in\mathcal A$ and
let $\phi_{\mathcal A}$ be a state on A, such that $0 <
\alpha=\phi_{\mathcal A}(p) < 1.$ Let $\mathcal B$ be a unital
C$^*$-algebra with a state $\phi_{\mathcal B}$ and let $$ (\mathcal
D,\phi)=(\mathcal A,\phi_{\mathcal A})*_{red}  (\mathcal
B,\phi_{\mathcal B})
$$ Let $\mathcal D_1$ be the C$^*$-subalgebra of $\mathcal D$ generated by $\Bbb Cp+ (0 \oplus\mathcal  A_2)\subseteq \mathcal
A$ together with $\mathcal B$.   Then $p\mathcal D p$ is generated
by $p\mathcal D_1p$ and $\mathcal A_1\oplus 0\subseteq \mathcal A$,
which are free in $(p\mathcal Dp, \frac 1 \alpha \phi|_{p\mathcal
Dp} ),$ i.e.
$$
(p\mathcal D_1 p, \frac 1 \alpha \phi|_{p\mathcal D_1 p})*_{red}
(\mathcal A_1, \frac 1 \alpha\phi_{\mathcal A}|_{\mathcal A_1})
 \simeq(p\mathcal Dp, \frac 1 \alpha \phi|_{p\mathcal Dp} )\subseteq  \mathcal D.
$$
\end{lemma}

\begin{proposition}
Let $C(\Bbb T)$ be the unital C$^*$-algebra consisting all
continuous functions on the unit circle $\Bbb T$ and  $\tau$   a
faithful trace of $C(\Bbb T)$ induced by the Lesbeague measure on
$\Bbb T$. Let $\mathcal A_2$ and $\mathcal B\ne \Bbb C$ be unital
separable
   C$^*$-algebras with faithful traces $\tau_2$, and $\psi$
respectively. Let $\mathcal A=C(\Bbb T)\oplus \mathcal A_2$ with a
faithful trace $\phi$ given by $\phi=\alpha \tau+(1-\alpha)\tau_2$
for some $0<\alpha<1$. Then
$$
(\mathcal A,\phi)*_{red}(\mathcal B,\psi)
$$ is not a quasidiagonal C$^*$-algebra.
\end{proposition}
\begin{proof}
Let  $ (\mathcal D,\phi*\psi)\simeq (\mathcal
A,\phi)*_{red}(\mathcal B,\psi). $  By Lemma 4.1.2, there is a
unital C$^*$-subalgebra $\mathcal D_2\ne \Bbb C$ in $\mathcal D $,
such that $(C(\Bbb T),\tau)*_{red} (\mathcal D_2 ,\frac 1
{(\phi*\psi)(I_{\mathcal D_2})}(\phi*\psi)|_{\mathcal D_2})$ can be
embedded (not necessary to be unital) into $ (\mathcal
A,\phi)*_{red}( \mathcal B ,\psi) $. Combining with Proposition
4.1.1, we completed the proof.
\end{proof}

Recall a unital C$^*$-algebra $\mathcal A$ with a faithful trace
$\phi$ is diffuse if there is a unitary $u$ such that $\phi(u^n)=0$
for all $n\ne 0$, i.e. $u$ is a Haar unitary in $\mathcal A$.

\begin{definition}
Suppose that $\mathcal A$ is a unital C$^*$-algebra with a faithful
tracial state $\phi$. Then $(\mathcal A,\phi)$ is called {\em
partially diffuse} if there is a partial isometry $v$ in $\mathcal
A$ such that $vv^*=v^*v$ and $\phi(v^n)=0$ for all $n\ne 0$.
\end{definition}

\begin{theorem}
Suppose that $\mathcal A$ is a unital C$^*$-algebra with a faithful
tracial state $\phi$. Then the following are equivalent:
\begin{enumerate}
  \item $(\mathcal A,\phi)$ is partially diffuse;
  \item There is a unital C$^*$-subalgebra $\mathcal B$ of $\mathcal
  A$ such that $(\mathcal B,\frac 1{\phi(I_{\mathcal B})}\phi|_{\mathcal B})$ is diffuse. (Note we don't require that
  $\mathcal B$ contains the unit of $\mathcal A$.)
  \item There is a unital C$^*$-subalgebra $\mathcal C$ of $\mathcal
  A$ such that
  $$
      (\mathcal C, \frac 1{\phi(I_{\mathcal C})} \phi)\simeq (C(\Bbb
      T), \tau),
  $$ where   $C(\Bbb T)$ is the unital C$^*$-algebra consisting all
continuous functions on the unit circle $\Bbb T$ and  $\tau$ is a
faithful trace of $C(\Bbb T)$ induced by the Lesbeague measure on
$\Bbb T$.
\item There is a self-adjoint element $x$ in $\mathcal A$
satisfying:
\begin{enumerate}
 \item [] {\em Suppose that $X$ is the spectrum of $x$ in $\mathcal A$
 and $\mu$ is the Borel measure on $X$ induced from the trace $\phi$.
 Then there are real numbers $a<b$ in $X$ such that (i) $\mu|_{ X  \cap[a,b]} $ has
 no atom;   (ii) the distance between $X  \cap[a,b]$ and $X\setminus [a,b]$
 is
 larger than $0$.}
\end{enumerate}
\end{enumerate}
\end{theorem}

\begin{proof}
$(1)\Leftrightarrow  (2)\Leftrightarrow  (3)$ is obvious.
$(1)\Leftrightarrow (4)$ is by Lemma 4.2 in \cite{DHR}.
\end{proof}

\begin{proposition}
Let $\mathcal A$ and $\mathcal B\ne \Bbb C$ be unital separable
 C$^*$-algebras with faithful traces $\phi$, and $\psi$ respectively.
If $(\mathcal A,\phi)$ is partially diffuse,  then
$$
(\mathcal A,\phi)*_{red}(\mathcal B,\psi)
$$ is not a quasidiagonal C$^*$-algebra.
\end{proposition}

\begin{proof} Note that $\mathcal A$ is partially diffuse. By Theorem 4.1.2, there is a unital C$^*$-subalgebra $\mathcal C$ of $\mathcal
A$ such that
$$
(\mathcal C, \frac {1}{\phi(I_{\mathcal C})} \phi)\simeq (C(\Bbb
T),\tau),
$$ where   $C(\Bbb T)$ is the unital C$^*$-algebra consisting all
continuous functions on the unit circle $\Bbb T$ and  $\tau$ is a
faithful trace of $C(\Bbb T)$ induced by the Lesbeague measure on
$\Bbb T$. Let $p= I_{\mathcal C} $ and $q=I_{\mathcal A}-p$ be the
projections in $\mathcal A$. Then $\phi$ is a faithful trace on the
unital C$^*$-subalgebra $  \mathcal Cp +\Bbb C q$ of $\mathcal A$
and
$$
( \mathcal C p+\Bbb C q, \phi)*_{red} (\mathcal B,\psi)\subseteq
(\mathcal A,\phi)*_{red}(\mathcal B,\psi).
$$ By Proposition 4.1.2, we know that $$
(\mathcal A,\phi)*_{red}(\mathcal B,\psi)
$$ is not a quasidiagonal C$^*$-algebra.
\end{proof}

\begin{lemma}
  Let $\mathcal A=\Bbb C\oplus \Bbb C$ and $\mathcal B=\Bbb
C\oplus \Bbb C$  with faithful traces $\phi$, and $\psi$
respectively. Let $p=1\oplus 0$ be a projection in $\mathcal B$.
Then   $$ (p ((\mathcal A,\phi)*_{red}(\mathcal B,\psi) )p, \frac 1
{\psi(p)} (\phi*\psi)|_{p ((\mathcal A,\phi)*_{red}(\mathcal B,\psi)
)p })
$$ is partially diffuse.
\end{lemma}
\begin{proof}The  C$^*$-algebra $(\mathcal
A,\phi)*_{red}(\mathcal B,\psi)$ was totally determined in Theorem
13 of \cite{ABH} (see also Proposition 2.7 in \cite{Dyk}). Thus the
structure of $p(\mathcal A,\phi)*_{red}(\mathcal B,\psi)p $ is also
determined as listed in Theorem 13 of \cite{ABH}. Now the rest
follows from Lemma 4.2 in \cite{DHR} (see also the proof of Lemma
4.1 in \cite{Dyk}).
\end{proof}
\begin{lemma}
Let $\tau_1,\tau_2$ and $\psi$ be faithful traces on the C$^*$
algebras $\mathcal A_1=\Bbb C\oplus \Bbb C$, $\mathcal A_2=\Bbb
C\oplus\Bbb C\oplus\Bbb C$ and $\mathcal A_3= \mathcal M_2(\Bbb C)$
respectively. Then\begin{enumerate}
  \item [(i)] $(\mathcal A_1,\tau_1)*_{red}(\mathcal A_2,\tau_2)=( \Bbb C\oplus \Bbb C,\tau_1)*_{red} ( \Bbb C\oplus \Bbb C\oplus \Bbb
  C,\tau_2)$ is not a quasidiagonal C$^*$-algebra;
  \item [(ii)] $(\mathcal A_1,\tau_1)*_{red}(\mathcal A_3,\psi)= ( \Bbb C\oplus \Bbb C,\tau_1)*_{red} ( \mathcal M_2(\Bbb C),\psi)$ is not a quasidiagonal
  C$^*$-algebra.
\end{enumerate}
\end{lemma}
\begin{proof}
(i) Let $\mathcal B=\Bbb C\oplus \Bbb C\oplus 0\subset \mathcal
A_2$ be a C$^*$-subalgebra of $\mathcal A_2$.   Let $p=1\oplus
1\oplus 0$ and $q= 0\oplus 0\oplus 1$ be projections in $\mathcal
A_2$. Let $\mathcal D_1$ be the C$^*$-subalgebra generated by
$\mathcal A_1$ and $\Bbb Cp+ \Bbb
  Cq$ in $(\mathcal A_1,\tau_1)*_{red} ( \mathcal
  A_2,\tau_2)$. Then
$$
( \mathcal A_1,\tau_1)*_{red} ( \mathcal
  A_2,\tau_2) \supseteq (\mathcal D_1,\tau_1*\tau_2|_{\mathcal D_1} )\simeq \left  ( \Bbb C\oplus \Bbb
C,\tau_1)*_{red} ( \Bbb Cp+ \Bbb
  Cq,\tau_2 \right );
$$
and, by Lemma 4.1.2, we have
\begin{enumerate}\item[ ] {\em   Fact 1:  $p\mathcal D_1p$ and $\mathcal B\oplus 0$ are free in $p\left ( ( \mathcal A_1,\tau_1)*_{red} ( \mathcal
  A_2,\tau_2)\right )p$ with respect to \\ $\frac 1 {\tau_2(p)}(\tau_1*\tau_2)|_{p\left (( \mathcal A_1,\tau_1)*_{red} ( \mathcal
  A_2,\tau_2)\right )p}$. }\end{enumerate}  By
Lemma 4.1.3, we know that  $(p\mathcal D_1p,\frac 1
{\tau_2(p)}(\tau_1*\tau_2)|_{p\mathcal D_1p})$ is partially diffuse.
Note that $\mathcal B\ne \Bbb C$. Combining with
  Proposition 4.1.3 and Fact 1, we know that the C$^*$-subalgebra generated $p\mathcal D_1p$ and $\mathcal B\oplus 0$ in $( \mathcal A_1,\tau_1)*_{red} ( \mathcal
  A_2,\tau_2)$ is not quasidiagonal. Hence    $( \mathcal A_1,\tau_1)*_{red} ( \mathcal
  A_2,\tau_2)$ is not quasidiagonal.

  (ii) Note $\mathcal A_1=\Bbb C\oplus \Bbb C$. Let
  $$
  u_1=\begin{pmatrix}
   1 &0\\ 0&-1
  \end{pmatrix} \qquad and \qquad  u_2=\begin{pmatrix}
   0 &1\\ 1&0
  \end{pmatrix}
  $$ be unitaries in $\mathcal M_2(\Bbb C)$. Then  $\mathcal A_1$,
  $u_1\mathcal A_1 u_1^*$ and $u_2\mathcal A_1 u_2^*$ are free in $( \Bbb C\oplus \Bbb C,\tau_1)*_{red} ( \mathcal M_2(\Bbb
  C),\psi)$.   Let
   $\mathcal B$ be C$^*$-subalgebra generated by  $\mathcal A_1$ and
  $u_1\mathcal A_1 u_1^*$  in $( \Bbb C\oplus \Bbb C,\tau_1)*_{red} ( \mathcal M_2(\Bbb
  C),\psi)$.
Then
$$
  \mathcal B\simeq (\mathcal A_1,\tau_1)*_{red}(\mathcal
  A_1,\tau_1)=(\Bbb C\oplus \Bbb C,\tau_1)*_{red}  (\Bbb C\oplus \Bbb
  C,\tau_1);
$$ and
\begin{enumerate}\item[ ] {\em  Fact 2: $\mathcal B$ and
$u_2\mathcal A_1 u_2^*$ are free in $( \Bbb C\oplus \Bbb
C,\tau_1)*_{red} ( \mathcal M_2(\Bbb
  C),\psi)$.} \end{enumerate}
    By
Lemma 4.1.3, $\mathcal B$ is partially diffuse.  Combining with
Proposition 4.1.3 and Fact 2, we know that the C$^*$-subalgebra
generated $\mathcal B$ and $u_2\mathcal A_1 u_2^*$   in $( \Bbb
C\oplus \Bbb C,\tau_1)*_{red} ( \mathcal M_2(\Bbb
  C),\psi)$ is not quasidiagonal.
  Hence
  $( \Bbb C\oplus \Bbb C,\tau_1)*_{red} ( \mathcal M_2(\Bbb
  C),\psi)$  is not quasidiagonal.
\end{proof}

The following proposition follows directly from preceding lemma.
\begin{proposition}
Suppose that $\mathcal A_1$  and $\mathcal A_2$  are unital
separable C$^*$-algebras with faithful tracial states $\tau_1$, and
$\tau_2$ respectively. If there are C$^*$-subalgebras $I_{\mathcal
A_i}\in\mathcal B_i\subseteq \mathcal A_i$ for $i=1,2$ such that (i)
$\mathcal B_1\simeq \Bbb C\oplus \Bbb C$; and (ii) either $\mathcal
B_2\simeq \Bbb C\oplus \Bbb C\oplus \Bbb C$ or $\mathcal B_2\simeq
\mathcal M_2(\Bbb C),$ then
$$
 (\mathcal A_1,\tau_1)*_{red}(\mathcal A_2,\tau_2)
$$ is not a quasidiagonal C$^*$-algebra.
\end{proposition}

We are are ready to show the following statement.
\begin{theorem}
Suppose that $\mathcal A_1$  and $\mathcal A_2$  are unital
separable AF algebras with faithful tracial states $\tau_1$, and
$\tau_2$ respectively. If  $dim_{\Bbb C}\mathcal A\ge 2$ and
$dim_{\Bbb C}\mathcal A_2\ge 3$, then
$$
 (\mathcal A_1,\tau_1)*_{red}(\mathcal A_2,\tau_2)
$$ is not a quasidiagonal C$^*$-algebra.
\end{theorem}

\begin{proof}Note that both $\mathcal A_1$ and $\mathcal A_2$ are
unital AF algebras.  Since $dim_{\Bbb C}\mathcal A_1\ge 2$, there is
a C$^*$-subalgebra $I_{\mathcal A_1}\in\mathcal B_1$ of $\mathcal A$
such that $\mathcal B_1\simeq \Bbb C\oplus \Bbb C$. Since $dim_{\Bbb
C}\mathcal A_2\ge 3$, there is a C$^*$-subalgebra $I_{\mathcal
A_2}\in\mathcal B_2$ of $\mathcal A_2$ such that either $\mathcal
B_2\simeq \Bbb C\oplus \Bbb C\oplus \Bbb C$ or $\mathcal B_2\simeq
\mathcal M_2(\Bbb C).$ Now it follows from Proposition  4.1.4, we
know that $
 (\mathcal A_1,\tau_1)*_{red}(\mathcal A_2,\tau_2)
$  is not a quasidiagonal C$^*$-algebra.
\end{proof}

\subsection{BDF extension semigroups of reduced free products of AH
algebras} Suppose $\mathcal A$ is a separable unital C$^*$-algebra.
The invariant $Ext(\mathcal A)$ was introduced by Brown, Douglas and
Fillmore in \cite{BDF}. $Ext(\mathcal A)$ is the set of equivalence
classes $[\pi]$ of  unital $*$-monomorphisms $\pi:\mathcal
A\rightarrow \mathcal{C}(\mathcal H)$, where $\mathcal{C}(\mathcal
H)=B(\mathcal H)/\mathcal K(\mathcal H)$ is the Calkin algebra for a
separable Hilbert space $\mathcal H=l^2(\Bbb Z)$. The equivalence
relation is defined as follows:
$$
\pi_1\sim\pi_2\Leftrightarrow \exists u\in\mathcal U(B(\mathcal H))
\ such \ that \ \forall a\in\mathcal A:
\pi_1(a)=\rho(u)\pi_2(a)\rho(u)^*,
$$ where $\mathcal U(B(\mathcal
H))$ is the unitary group of $B(\mathcal H)$ and $\rho:B(\mathcal
A)\rightarrow \mathcal C(\mathcal H)$ is the quotient map. There is
a natural semigroup structure on $Ext(\mathcal A)$. By a result of
Voiculescu, $Ext(\mathcal A)$ always has a unit. By a result of Choi
and Effros, $Ext(\mathcal A)$ is a group for every separable unital
nuclear C$^*$-algebras $\mathcal A$.  In \cite{Haag}, Haagerup and
Thorbj{\o}rnsen solved a long standing open problem by showing that
$Ext(C_r^*(F_2))$ is not a group.

In this subsection, we consider the BDF extension semigroups of
reduced free products of some unital AH algebras. First we recall a
useful fact, which  can be found in \cite{Br}, \cite{Haag} and
\cite{VoiQusi}. (See also Lemma 2.4 in \cite{HaSh5})
\begin{lemma}
Suppose that $\mathcal A$ is a unital separable MF algebras. If
$\mathcal A$ is not quasidiagonal, then $Ext(\mathcal A)$ is not a
group.
\end{lemma}
By Theorem 3.3.2  (or Theorem 3.4.1), Theorem 4.1.1 and Lemma 4.2.1,
we have the following result.
\begin{theorem}
Suppose that $\mathcal A_1$  and $\mathcal A_2$   are unital
separable  AH (or  ASH) algebras with faithful tracial states
$\tau_1$, and $\tau_2$ respectively. If $\mathcal A_1$ and $\mathcal
A_2$ satisfy Avitzour's condition, {\em  i.e. there are unitaries
$u\in \mathcal A_1$ and $v,w\in\mathcal A_2$ such that $$
\tau_1(u)=\tau_2(v)=\tau_2(w)=\tau_2(w^*v)=0,
$$}{then
$$
Ext \big ((\mathcal A_1,\tau_1)*_{red}(\mathcal A_2,\tau_2) \big )
\qquad \text{is not a group.}
$$
}
\end{theorem}
By Theorem 3.3.2 (or Theorem 3.4.1), Proposition 4.1.3 and Lemma
4.2.1, we have the following result.
\begin{theorem}
Let $\mathcal A$ and $\mathcal B\ne \Bbb C$ be unital separable AH
(or  ASH) algebras with faithful traces $\phi$, and $\psi$
respectively. If $\mathcal A$ is partially diffuse in the sense of
Definition 4.1.1, then
$$
Ext \big ((\mathcal A,\phi)*_{red}(\mathcal B,\psi) \big ) \qquad
\text{is not a group.}
$$
\end{theorem}

By Theorem 3.3.2  (or Theorem 3.4.1), Theorem 4.1.3 and Lemma 4.2.1,
we have the following result.
\begin{theorem}
Suppose that $\mathcal A $  and $\mathcal B$  are unital separable
AF algebras with faithful tracial states $\phi$, and $\psi$
respectively. If  $dim_{\Bbb C}\mathcal A\ge 2$ and $dim_{\Bbb
C}\mathcal B\ge 3$, then
$$
Ext \big ((\mathcal A,\phi)*_{red}(\mathcal B,\psi) \big )  \qquad
\text{is not a group.}
$$
\end{theorem}

\begin{example}
Let $\mathcal A$ and $\mathcal B $ be irrational C$^*$-algebras, or
UHF algebras, with faithful traces $\phi$, and $\psi$ respectively.
 Then
$ Ext \big ((\mathcal A,\phi)*_{red}(\mathcal B,\psi) \big ) \
\text{is not a group.} $
\end{example}

\section{Reduced Free Products of Tensor Products of Unital
C$^*$-algebras}

In this section, we will discuss some generalizations of the results
we obtained in the previous sections. Most of the results obtained
in this section are parallel to the ones in section 3 and their
proofs are also similar. Thus we skip most of the proofs of the
results in this section and sketched them only if necessary.

 The following
notation will be used in this section. Suppose that $G$ is a
countable discrete group. We will denote $C_r^* (G)$ the reduced
group C$^*$-algebra of $G$ and $\tau_G$ the canonical tracial state
of $C_r^* (G)$.
\subsection{A class of MF algebras}
\begin{definition}
Let $\mathcal S$ be the set of all these pairs $(\mathcal
A,\phi)$ such that $\mathcal A$ is a separable unital
C$^*$-algebra and $\psi$ is  a faithful tracial state of $\mathcal
A$ satisfying $(\mathcal A,\psi)*_{red}(C_r^*(F_n),\tau_{F_n})$
is an MF algebra for every integer $n\ge 1$.
\end{definition}
By Theorem 3.3.3, we have the following result.
\begin{proposition}
Suppose that $\mathcal A$ is a unital separable AH algebra and
$\psi$ is a faithful trace of $\mathcal A$. Then
$$
  (\mathcal A,\psi)\in\mathcal S,
$$  where $\mathcal S$
is defined in Definition 5.1.1.
\end{proposition}
\subsection{Minimal tensor products of unital C$^*$-algebras with faithful traces} In
this subsection, we will recall the definition of minimal tensor
product  of two unital C$^*$-algebras when both C$^*$-algebras have
faithful traces.

 Suppose that $\mathcal A_i$, $i=1,2$, are
unital C$^*$-algebras with faithful traces $\psi_i$. Each element
$a_i$ in $\mathcal A_i$ corresponds to a vector $\hat a_i$ in
$\mathcal H_i=L^2(\mathcal A_i, \psi_i)$. Let
$$\rho_i:\mathcal A_i\rightarrow B(\mathcal H_i)=B(L^2(\mathcal
A_i,\psi_i))$$ be the GNS representation of $\mathcal A_i$ such that
$$\psi_i(a_i)=\langle \rho_i(a_i) \hat{I}_{\mathcal A_i},
\hat{I}_{\mathcal A_i}\rangle, \qquad \forall \ a_i\in\mathcal
A_i.$$ Then the C$^*$-subalgebra generated by
$$\{\rho_1(a_1)\otimes I_{\mathcal H_2},   I_{\mathcal H_1}\otimes
\rho_2(a_2) \ | \ a_i\in\mathcal A_i, i=1,2\}
$$ in $B(\mathcal H_1\otimes \mathcal H_2)$ is the minimal tensor
product of $\mathcal A$ and $\mathcal B$, and is denoted by
$\mathcal A\otimes_{min}\mathcal B$.

Moreover, there is a canonical vector state
$\psi=\psi_1\otimes_{min} \psi_2$ defined on $\mathcal
A\otimes_{min}\mathcal B$ as follows:
$$
\psi(T) = \langle T (\hat I_{\mathcal A_1}\otimes \hat I_{\mathcal
A_2}), \hat I_{\mathcal A_1}\otimes \hat I_{\mathcal A_2}\rangle,
\ \  \ \forall \ T\in \mathcal A\otimes_{min}\mathcal B.
$$ If both $\psi_1,\psi_2$ are faithful traces of $\mathcal A_i$, then
$\psi=\psi_1\otimes_{min} \psi_2$ is also a faithful trace of
$\mathcal A\otimes_{min}\mathcal B$ (for example see \cite{Avit}).
And,
$$(id_{\mathcal A\otimes_{min}\mathcal B},  {\mathcal H_1}\otimes
{\mathcal H_2}, \hat I_{\mathcal A_1}\otimes \hat I_{\mathcal
A_2})$$ is a GNS representation of $(\mathcal A\otimes_{min}\mathcal
B,\psi_1\otimes_{min} \psi_2)$.

 Using the discussion as above and following the same strategy as in
 Lemma 3.1.2, we can prove the following result, whose proof is
 skipped.

\begin{lemma} Suppose that $\mathcal A$ is a separable
unital C$^*$-algebras with a faithful trace  $\psi$. Let
$\mathcal H=L^2(\mathcal A,\phi)$. Suppose that $\mathcal B$ is a
finite dimensional C$^*$-algebras with a basis $1,b_1,\ldots,
b_{d-1}$, where $d$ is the complex dimension of $\mathcal B$.
Suppose that $\{\tau,\tau_\gamma\}_{\gamma=1}^\infty$ is a family
of faithful tracial states of $\mathcal B$ satisfying $$
\lim_{\gamma\rightarrow \infty}\tau_{\gamma}(b)=\tau(b) \qquad
\forall \ b\in\mathcal B .$$

Let $\Bbb C^d$ be a $d$-dimensional complex Hilbert space with an
orthonormal basis $e_1,\ldots, e_d$. Then there is a sequence of
faithful unital $*$-representations $\rho_{\tau}, \rho_{\tau_\gamma}
:\mathcal A\otimes_{min}\mathcal B\rightarrow B(\mathcal
H)\otimes_{min} \mathcal M_d(\Bbb C)$ of $\mathcal
A\otimes_{min}\mathcal B$ on $\mathcal H\otimes\Bbb C^d$ for
$\gamma=1,2\ldots$ such that
\begin{enumerate}
 \item [(i)] $(\rho_{\tau}, \mathcal H\otimes\Bbb C^d,\hat I_{\mathcal A}\otimes e_1)$
  and  $(\rho_{\tau_\gamma}, \mathcal H\otimes\Bbb C^d,\hat I_{\mathcal A}\otimes e_1)$
   are GNS representations
 of $(\mathcal A\otimes_{min}\mathcal B,\psi\otimes_{min}  \tau)$, and $ (\mathcal A\otimes_{min}\mathcal B, \psi\otimes_{min} \tau_\gamma)$
 respectively.
   \item [(ii)] For each $1\le i\le d-1$,
   $$
\lim_{\gamma\rightarrow \infty} \|\rho_{\tau_\gamma}(a\otimes
b_i)-\rho_{\tau}(a\otimes b_i)\|=0, \qquad \forall \ a\in
\mathcal A
   $$
   \end{enumerate}
\end{lemma}

The proof of the following result is similar to Lemma 3.1.3 and is
skipped.

\begin{lemma}

  Suppose that $\mathcal A_i$, $i=1,2,$ is a separable unital C$^*$-algebra
     with a faithful tracial state $\psi_i$. Suppose that $\mathcal B$
     is a finite dimensional C$^*$-algebra with a family
     $\{\tau,\tau_{\gamma}\}_{\gamma=1}^\infty$ of faithful tracial
     states of $\mathcal B$ such that
     $$
           \lim_{\gamma\rightarrow
           \infty}\tau_{\gamma}(b)=\tau(b),\qquad \forall \
           b\in\mathcal B.
     $$

     Suppose that $x_1,\ldots, x_n$ is a family of elements in $ (\mathcal A_1,\psi)*_{red}(\mathcal
     B\otimes_{min}\mathcal A_2,\tau\otimes_{min} \psi_2)$.
     Then, for any $\epsilon>0$, there is a $\gamma_0>0$ such that
     $$
        \{ x_1,\ldots, x_n\}
         \subseteq_\epsilon  (\mathcal A_1,\psi)*_{red}(\mathcal
     B\otimes_{min}\mathcal A_2,\tau_\gamma\otimes_{min} \psi_2), \qquad \forall \ \gamma>\gamma_0.
     $$
\end{lemma}

\subsection{Some conclusions}

Suppose that $(\mathcal A,\psi)\in \mathcal S$, where $\mathcal S$
is defined in Definition 5.1.1. Then $(\mathcal
A,\psi)*_{red}(C_r^*(F_n),\tau_{F_n})$ is an MF algebra for all
$n\ge 2$. Consider an action $\alpha $ of $\Bbb Z_n$ on $(\mathcal
A,\psi)*_{red}(C_r^*(F_n),\tau_{F_n})$, induced by the following
mapping: {\em if $g$ is a natural generator of $\Bbb Z_n$ and
$u_1,\ldots, u_n$ are the natural generators of $C_r^*(F_n)$, then
$$
 \begin{aligned}
   \alpha(g) (x) &= x, \qquad \forall \ x\in \mathcal A;\\
   \alpha(g) (u_i) &= u_{i+1} \ for \ 1\le i\le n-1; \ = u_1 \ for \
   i=n.
 \end{aligned}
$$}

Using the same strategy as in the proof of Corollary 3.2.1,  we have
the following result.
\begin{lemma}
Suppose that $(\mathcal A,\psi)\in \mathcal S$, where $\mathcal S$
is defined in Definition 5.1.1. Then for all $n\ge 2$, $$(\mathcal
A\otimes_{min} C_r^*(\Bbb Z_n), \psi\otimes \tau_{\Bbb Z_n})*_{red}
(C_r^*(F_n),\tau_{F_n})$$ is an MF algebra.
\end{lemma}
Following the notation as above. Consider an action $\beta $ of
$\Bbb Z_n$ on $$(\mathcal A\otimes_{min} C_r^*(\Bbb Z_n),
\psi\otimes_{min} \tau_{\Bbb Z_n})*_{red} (C_r^*(F_n),\tau_{F_n}),$$
induced by the following mapping: {\em if $h$ is a natural generator
of $\Bbb Z_n$, then}
$$
 \begin{aligned}
   \beta(h) (x) &= x, \qquad\qquad \ \  \forall \ x\in \mathcal A;\\
   \beta(h)(   v) &= e^{2\pi i /n}   v, \qquad where \   v \ is \ a \ natural \ generator \ of \ C_r^*(\Bbb Z_n); \\
   \beta(h) (u_j) &= u_{j+1} \ \qquad\qquad for \ 1\le j\le n-1; \ \ \ \\ \beta(h) (u_n)&= u_1
   .
 \end{aligned}
$$
Modifying the proof of Lemma 3.2.3 slightly, we have the following
result.
\begin{lemma}
Suppose that $(\mathcal A,\psi)\in \mathcal S$, where $\mathcal S$
is defined in Definition 5.1.1. Then for all $n\ge 2$, $$ (\mathcal
A\otimes_{min}  \mathcal M_n(\Bbb C), \psi\otimes_{min} \tau_{ n})
*_{red} (\mathcal A\otimes_{min} \mathcal M_n(\Bbb C),
\psi\otimes_{min} \tau_{ n})
 $$ is a C$^*$-subalgebra of $$ \big ((\mathcal A\otimes_{min} C_r^*(\Bbb Z_n),
\psi\otimes \tau_{\Bbb Z_n})*_{red} (C_r^*(F_n),\tau_{F_n})\big
)\rtimes_{\beta,r}\Bbb Z_n; $$  and, therefore, is an MF algebra,
where $\mathcal M_n(\Bbb C)$ is $n\times n$ matrix algebra with a
trace $\tau_n$.
\end{lemma}

Combining Lemma 5.2.2, Lemma 5.3.2 and the strategy used in Theorem
3.3.1, we have the following result.
\begin{theorem}
Suppose that $(\mathcal A,\psi)\in \mathcal S$, where $\mathcal S$
is defined in Definition 5.1.1. Suppose that $\mathcal B_i$ is a
unital  AF algebra with a faithful trace $\phi_i$ for
$i=0,1,2,\ldots, n$. Then
$$(\mathcal A\otimes_{min} \mathcal B_0, \psi\otimes_{min}
\phi_0)*_{red}  (\mathcal A\otimes_{min}\mathcal
B_1,\psi\otimes_{min}\phi_1)
*_{red}\cdots *_{red}(\mathcal A\otimes_{min}\mathcal B_n,\psi\otimes_{min}\phi_n) $$ is an MF algebra.
\end{theorem}

By Lemma 3.3.3, we have the following result.
\begin{theorem}
Suppose that $(\mathcal A,\psi)\in \mathcal S$, where $\mathcal S$
is defined in Definition 5.1.1. Suppose that $\mathcal B_i$ is a
unital  AH algebra with a faithful trace $\phi_i$ for
$i=0,1,2,\ldots, n$. Then
$$(\mathcal A\otimes_{min} \mathcal B_0, \psi\otimes_{min}
\phi_0)*_{red}  (\mathcal A\otimes_{min}\mathcal
B_1,\psi\otimes_{min}\phi_1)
*_{red}\cdots *_{red}(\mathcal A\otimes_{min}\mathcal B_n,\psi\otimes_{min}\phi_n) $$ is an MF algebra. In particular,
for every $n\ge 2$,
$$(\mathcal A\otimes_{min} \mathcal B_0, \psi\otimes_{min}
\phi_0)*_{red} (C_r^*(F_n),\tau_{F_n})$$ is an MF algebra. I.e.
$$
(\mathcal A\otimes_{min} \mathcal B_0, \psi\otimes_{min} \phi_0) \in
\mathcal S.
$$
\end{theorem}
\begin{corollary}
For $i=1,2$, let  $ \mathcal A_1^{(i)},\ldots, \mathcal
A_n^{(i)},\mathcal B^{(i)} $  be a family of unital   AH algebras
with faithful tracial states $  \psi_1^{(i)},\ldots, \psi_n^{(i)},
\phi^{(i)}  $, respectively.  Let
$$
\begin{aligned}
   (\mathcal A^{(i)},\psi^{(i)})=  (\mathcal
   A_1^{(i)},\psi^{(i)}_1)*_{red}\cdots *_{red} (\mathcal
   A_n^{(i)},\psi^{(i)}_n), \qquad for \ i=1,2;
\end{aligned}
$$ and $\psi^{(i)}\otimes_{min} \phi^{(i)}$ be a faithful trace on $
 \mathcal A^{(i)}\otimes_{min} \mathcal B^{(i)} $. Then
 $$
( \mathcal A^{(1)}\otimes_{min} \mathcal B^{(1)} ,
\psi^{(1)}\otimes_{min} \phi^{(1)})*_{red}(\mathcal
A^{(2)}\otimes_{min} \mathcal B^{(2)} , \psi^{(2)}\otimes
\phi^{(2)})
 $$ is an MF algebra.
\end{corollary}
\begin{proof}
Let $$(\mathcal A,\tau)=(\mathcal A^{(1)},\psi^{(1)})*_{red}
(\mathcal A^{(2)},\psi^{(2)}).$$ Let $\tau\otimes_{min}
\phi^{(1)}\otimes_{min} \phi^{(2)}$ be a faithful tracial state on
$\mathcal A\otimes_{min}\mathcal B_1\otimes_{min} \mathcal B_2$. By
Theorem 3.3.3 and Theorem 5.3.2, we know that $$(\mathcal A,\tau)\in
\mathcal S;$$ and
$$
\mathcal D=(\mathcal A\otimes_{min}\mathcal B_1\otimes_{min}
\mathcal B_2, \tau\otimes_{min} \phi^{(1)}\otimes_{min}
\phi^{(2)})*_{red} (C_r^*(F_2),\tau_{F_2})$$  { is an MF algebra.}
Therefore, embedded as a C$^*$-subalgebra of $\mathcal D$,
$$ ( \mathcal A^{(1)}\otimes_{min} \mathcal B^{(1)} ,
\psi^{(1)}\otimes_{min} \phi^{(1)})*_{red}(\mathcal
A^{(2)}\otimes_{min} \mathcal B^{(2)} , \psi^{(2)}\otimes_{min}
\phi^{(2)})
 $$ is an MF algebra.
\end{proof}

\begin{example}
Suppose that $\mathcal A_i$, $i=1,2$, is an irrational
C$^*$-algebra, or a UHF algebra, with a faithful tracial state
$\psi_i$. For all $m, n\ge 1$, let
$$\mathcal D=(C_r^*(F_m)\otimes_{min} \mathcal A_1, \tau_{F_m}\otimes_{min}
\psi_1)*_{red} (C_r^*(F_n)\otimes_{min} \mathcal A_2,
\tau_{F_n}\otimes_{min} \psi_2).$$ Then $\mathcal D$ is an MF
algebra and $Ext(\mathcal D)$ is not a group.
\end{example}

\vspace{1cm}

{\small

  \noindent {Don Hadwin} \makebox[4.5cm] { }
\makebox[5cm] {Jiankui Li } \\
 \makebox[0.8cm] {} \makebox[5cm] {Department of Mathematics \& Statistics}
\makebox[2.1cm] { }
\makebox[5cm] {Department of Mathematics} \\
  \makebox[5cm] {University of New Hampshire} \makebox[4.8cm]
{ } \makebox[5cm] {East China University of Science and Technology} \\
  {Durham, NH, 03824} \makebox[3cm] { }
\makebox[6.2cm] {Shanghai China} \\
  {email: don@math.unh.edu} \makebox[3cm] { }
\makebox[6.2cm] {email: jiankuili@yahoo.com } \\

\vspace{0.1cm}

 \noindent {Junhao Shen} \makebox[ 4.6cm] { }
\makebox[5cm] {Liguang Wang} \\
 \makebox[0.8cm] { } \makebox[5cm] {Department of Mathematics \& Statistics  }
\makebox[2.3cm] { }
\makebox[5cm] {School of Mathematical Science  } \\
  \makebox[5cm] {University of New Hampshire } \makebox[2.6cm]
{ } \makebox[5cm] {Qufu Normal University} \\
  {Durham, NH, 03824} \makebox[3.4cm] { }
\makebox[6.6cm] {Qufu, Shandong, China} \\
  {email: jog2@cisunix.unh.edu  } \makebox[3.1cm] { }
\makebox[6cm] {email: wangliguang0510@163.com} \\

}

\end{document}